\documentclass[12pt,A4,leqno]{amsart}
\usepackage{amsfonts}
\usepackage{mathrsfs}
\usepackage[T1]{fontenc}
\usepackage{amssymb,amscd}
\usepackage{upgreek}
\usepackage{color}
\usepackage{tikz-cd}
\usepackage{dsfont}
\usepackage{geometry}
\usepackage{epsf}	
\usepackage{graphicx}
\usepackage{extarrows}
\usepackage[curve]{xypic}
\usepackage{hyperref} %adding hyperlink
\usepackage{url}
\hypersetup{colorlinks,linkcolor={blue},citecolor={blue},urlcolor={blue}} %urlcolor={blue} 
\usepackage[capitalise]{cleveref}

\theoremstyle{plain}
\newtheorem{thm}{Theorem}[section]
\newtheorem{lem}[thm]{Lemma}
\newtheorem{prop}[thm]{Proposition}
\newtheorem{cor}[thm]{Corollary}

\theoremstyle{definition}
\newtheorem{defi}[thm]{Definition}
\newtheorem{exam}[thm]{Example}
\newtheorem{rem}[thm]{Remark}

\newcommand{\R}{\mathbb R}
\newcommand{\Z}{\mathbb Z}

\newcommand{\nn}{\vskip 0.2cm}
\newcommand{\n}{\vskip 0.1cm}

\renewcommand{\geq}{\geqslant}
\renewcommand{\leq}{\leqslant}

\makeatletter
\renewcommand{\@secnumfont}{\bfseries}
\makeatother

\makeatletter
\def\section{%
  \@startsection{section}{1}
    {\z@}
    {2.0ex plus 0.8ex minus .1ex}
    {1.0ex plus .2ex}
    {\bfseries\large\centering\MakeUppercase}%
}
\makeatother

\begin{document}
\title{Weighted Homology and Cohomology of Weighted Polyhedra}
\author{*Yin Wei, $^\dagger$Lisu Wu, and $^\ddagger$Li Yu}
\address{*School of Mathematics, Nanjing University\\ Nanjing\\210093\\ P.R.China}
\email{2986343993@qq.com}
\address{$^\dagger$College of Mathematics and Systems Science, Shandong University of 
Science and Technology, Tsingtao, 266590, P.R.China}
\email{wulisu@sdust.edu.cn} 
\address{$^\ddagger$School of Mathematics, Nanjing University\\ Nanjing\\210093\\ P.R.China}
\email{yuli@nju.edu.cn}


\date{\today}


\begin{abstract}
 We define the notion of weighted polyhedron which can be thought of as the geometric realization of a weighted simplicial complex introduced by Dawson~\cite{Daw90}. Moreover, we will define a weighted version of singular homology theory for a weighted polyhedron and prove that it is
isomorphic to the weighted simplicial homology of the weighted polyhedron. This implies that weighted simplicial homology is an invariant
  under isomorphisms and more generally under certain type of homotopy equivalences of weighted polyhedra. Moreover, we will generalize the cup product and cap product to weighted singular cohomology. In addition, we will interpret some known theories of orbifolds
   in terms of our weighted singular homology and cohomology.
\end{abstract}

\maketitle

\section{Introduction}

Let $K$ be a simplicial complex. The geometric realization $|K|$ of $K$ is often referred to as a polyhedron (for example in Munkres~\cite{Munk84}). 
A (positive) \emph{weight} on $K$ is a function 
$w$ assigning a positive integer to each simplex of $K$, which satisfies: for any simplices $\sigma, \sigma'$ in $K$,
  \[ \sigma'\ \text{is a face of}\ \sigma\ \Rightarrow\ w(\sigma')\, |\, w(\sigma).  \]
 The pair $(K,w)$ is called a \emph{weighted simplicial complex}. This notion was first introduced by Dawson in~\cite{Daw90} where it was 
 used to construct nonstandard homology theories for categories of a combinatorial nature, such as preconvexity spaces.
 \n
 
  For each $n\geq 0$, let $C_n(K)$ denote the free abelian group generated by all the oriented $n$-simplices of $K$.
  By abuse of notation, we will use the symbol $\sigma$ to denote a simplex or an oriented simplex in different occasions.
  Then the weight function $w$ determines a weighted
  boundary operator $\partial^w$ on $C_*(K)$ by:
  \begin{equation} \label{Equ:weighted-Boun-AW}
  \partial^{w}: C_n(K)\rightarrow C_{n-1}(K), \ \  \sigma \mapsto \sum^n_{j=0} (-1)^j \frac{w(\sigma)}{w(\partial_j(\sigma))}  \partial_j \sigma 
   \end{equation}
  where the face operators $\partial_j$ are defined as in the standard simplicial homology, i.e.
 for an oriented simplex $\sigma=[v_0,\cdots,v_n]$, $\partial_j\sigma=[v_0,\cdots,\widehat{v}_j,\cdots,v_n]$. It is easy to check that $\partial^w\circ\partial^w=0$. 
  The chain complex $(C_*(K), \partial^w)$ is called the \emph{weighted simplicial chain complex} of $(K,w)$, whose
    homology group is called the \emph{weighted simplicial homology} of $(K,w)$, denoted by $H_*(K,\partial^w)$.\n
    
 Weighted simplicial homology did not receive much attention in the earlier days. But in the recent years, some interesting applications of weighted simplicial homology 
 were found in computational topology and topological data analysis; see Ren, Wu and Wu~\cite{RenWuWu18, RenWuWu21},
  Wu, Ren, Wu and Xia~\cite{WuRenWuXia20} and
Baccini, Geraci and Bianconi~\cite{BGB22}.\n

 In this paper,  we introduce a new type of weighted simplicial complexes which are parallel to Dawson's.

   \begin{defi}
 A \emph{descending type weighted simplicial complex}
 is a pair $(L, \overline{w})$ where $L$ is a simplicial complex and $\overline{w}$ is a function assigning a positive integer to each simplex of $L$, which satisfy: for any simplices $\sigma, \sigma'$ in $L$,
  \[ \sigma'\ \text{is a face of}\ \sigma \Rightarrow \overline{w}(\sigma)\, |\, \overline{w}(\sigma').  \]
  
  We call $\overline{w}$ a \emph{descending weight} on $L$. Similarly to~\eqref{Equ:weighted-Boun-AW},
  we define a weighted boundary operator $\partial^{\overline{w}}: C_n(L)\rightarrow C_{n-1}(L)$ for every $n\geq 0$ by:
  \begin{equation} \label{Equ:weighted-Boun-DW}
   \partial^{\overline{w}}: \sigma \longmapsto \sum^n_{j=0} (-1)^j \frac{\overline{w}(\partial_j(\sigma))}{\overline{w}(\sigma)}  \partial_j \sigma. 
   \end{equation}
  It is easy to show that $\partial^{\overline{w}}\circ\partial^{\overline{w}}=0$.  We call $(C_*(L), \partial^{\overline{w}})$ the
  \emph{weighted simplicial chain complex} of $(L,\overline{w})$, whose homology group is called the \emph{weighted simplicial homology group} of $(L,\overline{w})$, denoted by 
  $H_*(L,\partial^{\overline{w}})$.
 \end{defi}

  Descending type weighted simplicial complexes also arise naturally in geometry and topology. For example, any orbifold canonically determines some
     descending type weighted simplicial complexes (see Section~\ref{Sec:Relation} for the details).\n

 Since now there are two different types of weighted simplicial complexes, we use the following conventions in the rest of the paper.
\begin{itemize}
\item[\textbf{(CV-1)}] To avoid ambiguity, we call the weighted simplicial complex $(K,w)$ defined in Dawson~\cite{Daw90} an \emph{ascending type weighted simplicial complex} and call $w$ an \emph{ascending weight}.
\n

\item[\textbf{(CV-2)}] When we say $(K,\mu)$ is a weighted simplicial complex, $\mu$ could be either an ascending weight or a descending one. Moreover, we say that an ascending weight and a descending weight on $K$ are of \emph{opposite type}, so do we call the corresponding weighted simplicial complexes.\n

\item[\textbf{(CV-3)}] Let $(C_*(K,\mu),\partial^{\mu})$ denote the
 weighted simplicial chain complex of $(K,\mu)$.
 The weighted boundary map $\partial^{\mu}$ is defined by (see~\eqref{Equ:weighted-Boun-AW} and~\eqref{Equ:weighted-Boun-DW})
 \begin{align} \label{Equ:Boundary-Unified}
   \partial^{\mu}: C_n(K) \rightarrow C_{n-1}(K), \ 
   \sigma  \mapsto   \sum^n_{j=0} (-1)^j c^{\mu}_j(\sigma)   \partial_j \sigma 
\end{align}
where $c^{\mu}_j(\sigma) = \begin{cases}
\displaystyle \frac{\mu(\sigma)}{\mu(\partial_j(\sigma))},  &  \text{if $\mu$ is ascending}; \\
\displaystyle   \overset{\ \ }{\frac{\mu(\partial_j(\sigma))}{\mu(\sigma)}},  &  \text{if $\mu$ is descending}.
 \end{cases} 
  $\n
  
   \item[\textbf{(CV-4)}] If $(K,\mu)$ is a weighted simplicial complex and $L$ is a subcomplex of $K$, we also use $(L,\mu)$ to denote the weighted simplicial complex
  $(L,\mu|_L)$ and, use $H_*(L,\partial^{\mu})$
  instead of the cumbersome notation $H_*(L,\partial^{\mu|_L})$ to denote its weighted homology groups.  
 \end{itemize}
 \n
 
 To visualize a weighted simplicial complex
 $(K,\mu)$, it is natural to label
   each point in the \emph{geometric realization}
  $|K|$ of $K$ by some value defined by the weight $\mu$. 
  This leads to the following definitions. 
 
  \begin{defi}[Geometric Realization of Weighted Simplicial Complex] \label{Def:Realization-WSC}\ \n
  Any weighted simplicial complex $(K,\mu)$ determines
  a map $\lambda_{\mu}: |K| \rightarrow \Z_+$ by:
      \begin{equation} \label{Equ:lambda-mu-relation}
         \lambda_{\mu} (x) = \mu(\mathrm{Car}_K(x)), \ \forall  x\in |K|,
    \end{equation}
   where $\mathrm{Car}_K(x)$ is the unique simplex of $K$
   such that $|\mathrm{Car}_K(x)|$ contains $x$ in the relative interior.  
     We call $(|K|,\lambda_{\mu})$ the \emph{geometric realization} of $(K,\mu)$.
  \end{defi}

  \begin{defi}[Weighted Space]
    A \emph{weighted space} is a pair $(X,\lambda)$ where
  $X$ is a topological space and $\lambda: X\rightarrow \Z_+$.
 Two weighted spaces $(X,\lambda)$ and $(X',\lambda')$ are called \emph{isomorphic} if there exists a homeomorphism $f: X\rightarrow X'$ such that $f$ is weight-preserving, i.e. $\lambda'(f(x))=\lambda(x)$ for any $x\in X$.
  \end{defi}

    \begin{defi}[Weighted Polyhedron] \label{Def:Weighted-Polyhedron} 
    A weighted space $(X,\lambda)$ is called a 
     \emph{weighted polyhedron} if there exists a
     weighted simplicial complex $(K,\mu)$ such that $(X,\lambda)$ is isomorphic
    to $(|K|,\lambda_{\mu})$, and
    we call $(K,\mu)$  a \emph{weighted triangulation} of $(X,\lambda)$. 
   In addition, we say that a weighted polyhedron $(X,\lambda)$ is of
 \emph{ascending type} (or \emph{descending type}) 
    if its weighted triangulation is an ascending type (or descending type) weighted simplicial complex.
\end{defi}
  \n

 Unlike ordinary simplicial homology, a weighted simplex may have nontrivial weighted simplicial homology in dimension greater than $0$. But in many occasions, it is not convenient to have such kind of simplices in a weighted triangulation. 
 The following type of weighted simplicial complexes 
 were introduced in~\cite{Daw90} which plays an important role in our study.\n
  
\begin{defi}[Divisibly Weighted Simplicial Complex]
 A simplex $\sigma$ in a weighted simplicial complex $(K,\mu)$ is called \emph{divisibly weighted} if all the vertices of $\sigma$ can be ordered as
 $\{v_0,\cdots, v_k\}$ such that
   \begin{equation} \label{Equ:Weight-cond}
     \mu(v_0)\, |\, \cdots\, | \, \mu(v_k).
   \end{equation}
  If all the simplices in $(K,\mu)$ are divisibly weighted, 
  the weight function $\mu$ is called \emph{divisible}
  and $(K,\mu)$ is called a \emph{divisibly weighted simplicial complex}.  
  \end{defi}
   
  The above definition is defined for ascending type weights in~\cite{Daw90}, but it clearly makes sense for descending type weights as well.
 Notice that a divisible weight $\mu$ on $K$ is completely determined by its values on the vertices of $K$. Indeed, for a simplex $\sigma$ whose vertices are ordered
  as in~\eqref{Equ:Weight-cond},
  $$ \mu(\sigma)= \begin{cases}
   \mu(v_k),  &  \text{if $\mu$ is ascending}; \\
   \mu(v_0),  &  \text{if $\mu$ is descending}.
 \end{cases} $$

  We will prove that a weighted polyhedron
  $(X,\lambda)$ always
  admits a weighted triangulation $(K,\mu)$ where $\mu$
  is divisible (see Corollary~\ref{Cor:Div-Weight-Triangl}).  We call such
  $(K,\mu)$ a \emph{divisibly weighted triangulation} of $(X,\lambda)$. 
     
   \begin{defi}[Weighted simplicial homology] \label{Defi:Pseudo-Orbi-Homol} \ \n
  
  Suppose $(X,\lambda)$ is a weighted polyhedron and let 
   $(K,\mu)$ be a divisibly weighted triangulation of $(X,\lambda)$. We call $H_*(K,\partial^{\mu})$    
     the \emph{weighted simplicial homology} of $(X,\lambda)$, denoted by $H^W_*(X,\lambda)$.
 \end{defi}
 
The following theorem shows that the weighted simplicial homology of a weighted polyhedron is 
well-defined.
 
 \begin{thm} \label{Thm:Invariance}
   For any divisibly weighted triangulations $(K,\mu)$ and $(K',\mu')$ of a weighted polyhedron $(X,\lambda)$, their weighted simplicial homology $ H_*(K, \partial^{\mu})$  and $H_*(K', \partial^{\mu'})$ are isomorphic.
   \end{thm}

   To prove this theorem, we introduce the weighted singular homology $\mathcal{H}^W_*(X,\lambda)$ of a weighted polyhedron $(X,\lambda)$,
   which is the natural generalization of singular homology to weighted spaces.
     We will prove in Theorem~\ref{Thm:Main-Isom} that $\mathcal{H}^W_*(X,\lambda)$ is isomorphic to $H_*(K,\partial^{\mu})$ for any divisibly weighted triangulation $(K,\mu)$ of $(X,\lambda)$, which then implies
     Theorem~\ref{Thm:Invariance}. In addition, there is another way to prove the well-definedness of weighted simplicial homology which was
      given by the authors in~\cite{WeiWuYu21}.\n

Moreover, we can define the cohomology theory associated to
the weighted simplicial homology. For a weighted simplicial complex  $(K,\mu)$ and an abelian group $G$, the \emph{weighted simplicial cochain complex} with $G$-coefficients is obtained by applying
  the $\mathrm{Hom}(-, G)$ functor to 
  the chain complex $(C_*(K),\partial^{\mu})$, denoted by $(C^*(K;G),\delta^{\mu})$,
  where $\delta^{\mu}$ is the coboundary map determined by
    $\partial^{\mu}$.
  Then the \emph{weighted simplicial cohomology group of $(K,\mu)$} with $G$-coefficients is
   $$ H^*(K,\delta^{\mu};G):= H^*((C^*(K;G),\delta^{\mu})).$$ 
   
  In particular, if $(K,\mu)$ be a divisibly weighted triangulation of a weighted polyhedron 
  $(X,\lambda)$, we call $H^*(K,\delta^{\mu};G)$ the \emph{weighted simplicial cohomology} of $(X,\lambda)$ with $G$-coefficients, denoted by $H^*_{W}(X,\lambda;G)$.
Similarly to Theorem~\ref{Thm:Invariance}, $H^*_{W}(X,\lambda;G)$ is independent on which divisibly weighted triangulation of $(X,\lambda)$ used in the definition.
In addition, we can use the universal coefficient theorem  (see~\cite[Theorem 3.2]{Hatcher02}) to compute the weighted simplicial cohomology from the weighted simplicial homology.  \n

Although the definitions of weighted simplicial cohomology for ascending type and descending type weighted polyhedra are parallel, 
 their properties are not.
 We will see in Section~\ref{Sec:Product-Cohomology} that there is a natural graded commutative
   product structure on weighted simplicial cohomology
   for descending type weighted polyhedra, but there is no such a product on ascending type weighted polyhedra.
  \n
 
 The primary examples of weighted polyhedra in our paper come from orbifolds. For an orbifold $\mathcal{M}=(M,\mathcal{U})$ where $\mathcal{U}$ is an orbifold atlas on $M$, define a weight function $\lambda_{\mathcal{M}}: M\rightarrow \Z_+$ by: for any $x\in M$,
  $$\lambda_{\mathcal{M}}(x)=\text{
 the order $|G_x|$ of the isotropy group $G_x$ of $x$}.$$ 
 It is natural to think of $(M,\lambda_{\mathcal{M}})$ as a descending type weighted polyhedron (see Proposition~\ref{Prop:Weight-Equal}). We will prove that the weighted singular homology and cohomology of $(M,\lambda_{\mathcal{M}})$ are equivalent to some known theories of orbifolds. \n

     The paper is organized as follows. In Section~\ref{Sec:Basic-Construc}, we review some basic facts and constructions related to weighted simplicial complexes. In Section~\ref{Sec:Bary-Subdiv}, we study the barycentric subdivision of a weighted simplicial complex.  In Section~\ref{Sec:Weigh-Singul-Homology}, we introduce   the weighted singular homology of a weighted polyhedron and 
     use it to prove Theorem~\ref{Thm:Invariance}. Our proof uses the generalized Mayer-Vietoris sequence and
     the \v{C}ech complex of a cosheaf. 
 In Section~\ref{Sec:Product-Cohomology}, we construct a natural graded commutative product on the weighted singular cohomology of a descending type weighted polyhedron, which generalizes the cup product in ordinary cohomology theory. Moreover, this leads to a weighted cap product between the weighted singular homology and weighted singular cohomology.  In Section~\ref{Sec:Relation}, we study weighted polyhedra that come from orbifolds and discuss the relations between the weighted singular homology and cohomology of orbifolds with some other known orbifold theories.
    In Section~\ref{Sec:Example}, we compute some examples to demonstrate the weighted singular (or simplicial) homology. In the appendix, we write a proof of a basic but important lemma that is used to show the isomorphism 
    between weighted singular homology and weighted simplicial homology.

 \vskip 0.4cm
 
 \section{Basic constructions of weighted simplicial complexes} \label{Sec:Basic-Construc}
 
  We review some basic constructions and facts of weighted simplicial complexes from~\cite{Daw90}, which make sense for weighted simplicial complexes 
  of both types.

  \subsection{Cartesian product of weighted simplicial complexes}
 \ \n
  
   Let $(K,\mu)$ and $(K',\mu')$ be two weighted simplicial complexes and let
  $$v_1 < \cdots < v_m, \ \ v'_1 < \cdots < v'_{n}$$
  be some total orderings of the vertices of $K$ and
  $K'$, respectively. Then the Cartesian product $K\times K'$ is a simplicial complex whose vertex set is 
  $$\{ (v_i,v'_j)\,|\, 1\leq i \leq m, 1\leq j \leq n \}.$$
  Moreover, all the simplices in $K\times K'$ are of the form  
  \begin{equation} \label{Equ:Product-Simplices}
    \{ (v_{i_1},v'_{j_1}),\cdots, (v_{i_s},v'_{j_s}) \},
  \ i_1 \leq \cdots \leq i_s,\ j_1 \leq \cdots \leq j_s,
  \end{equation}
  where $\{v_{i_1},\cdots, v_{i_s}\}$ is a simplex in $K$
  and $\{ v'_{j_1},\cdots, v'_{j_s} \}$ is a simplex in $K'$.
  \n
  
 \begin{defi}[Cartesian product of weighted simplicial complexes]\label{Def:Product-Wt-Complexes}
   Let $(K,\mu)$ and $(K',\mu')$ be weighted simplicial complexes of the same type. 
  The \emph{Cartesian product} of $(K,\mu)$ and $(K',\mu')$ with respect to some total orderings of the vertices of $K$ and $K'$
   is a weighted simplicial complex $(K\times K',\mu\times \mu')$, where for a simplex
  $\{(v_{i_1},v'_{j_1}),\cdots, (v_{i_s},v'_{j_s})\}$ of
  $K\times K'$,
  \begin{equation} \label{Equ:Product-Weight}
   \mu\times \mu'\big( \{(v_{i_1},v'_{j_1}),\cdots, (v_{i_s},v'_{j_s})\} \big):= \mu(\{v_{i_1},\cdots, v_{i_s}\})\cdot \mu'(\{v'_{j_1},\cdots, v'_{j_s}\}).  
   \end{equation}
   It is easy to see that
  $(K\times K',\mu\times \mu')$ has the same type as $(K,\mu)$ and $(K',\mu')$.  
  \end{defi}
  
  Notice that the simplicial complex structure of $K\times K'$ depends on the ordering of vertices of $K$ and $K'$, so does $(K\times K',\mu\times \mu')$. But 
  we usually omit the ordering of vertices in our notation. 
  
  \begin{rem}
   The Cartesian product of weighted simplicial complexes defined here is a little different from the product of weighted simplicial complexes defined  in~\cite{Daw90} (see~\cite[Proposition 1.1]{Daw90}).
   \end{rem}

 \subsection{Morphisms between weighted simplicial complexes}\label{subsection-morphism}
 \ \n
 Let $(K,\mu)$ and $(K',\mu')$ be weighted simplicial complexes of the same type. A simplicial map $\varrho: K\rightarrow K'$ is called a \emph{morphism} from $(K,\mu)$ to $(K',\mu')$ if for every simplex $\sigma$ of $K$, we have
   \begin{equation} \label{Equ:Morphism}
    \begin{cases}
   \mu'(\varrho(\sigma)) \mid \mu(\sigma),  &  \text{if $\mu$ and $\mu'$ are ascending}; \\
   \mu(\sigma) \mid  \mu'(\varrho(\sigma)),  &   \text{if $\mu$ and $\mu'$ are descending}.
 \end{cases} 
   \end{equation} 
  For brevity, we write a morphism from $(K,\mu)$ to $(K',\mu')$ as $\varrho: (K,\mu)\rightarrow (K',\mu')$.
    We call $\varrho$ \emph{weight-preserving} if $\mu'(\varrho(\sigma))=\mu(\sigma)$ for any simplex $\sigma$ in $K$. Moreover, we call
     $\varrho$ an \emph{isomorphism} if $\varrho$ is a weight-preserving simplicial homeomorphism. 
      It is easy to check that a morphism  $\varrho: (K,\mu)\rightarrow (K',\mu')$ induces a chain map 
   \[ \varrho_{\#}:
   (C_*(K),\partial^{\mu})\rightarrow (C_*(K'), \partial^{\mu'}),  \]
   \[ \text{where} \ \varrho_{\#}(\sigma)= \begin{cases}
 \displaystyle \frac{\mu(\sigma)}{\mu'(\varrho(\sigma))} \varrho(\sigma),  &  \text{if $\mu$ and $\mu'$ are ascending}; \\
 \displaystyle  \overset{\ \ }{\frac{\mu'(\varrho(\sigma))}{\mu(\sigma)}} \varrho(\sigma) ,  &  \text{if $\mu$ and $\mu'$ are descending}.
 \end{cases}  \]
  Then $\varrho$ further induces a homomorphism on the weighted simplicial homology
  $$ \varrho_*:  H_*(K,\partial^{\mu})\rightarrow H_*(K',\partial^{\mu'}).$$
  
  From the above definitions, we obtain the following lemma immediately for divisibly weighted simplicial complexes.
   
    \begin{lem} \label{Lem:Morph-Vertex}
       Suppose $(K,\mu)$ and $(K',\mu')$ are two divisibly weighted simplicial complexes of the same type.  Then
   a simplicial map $\varphi: K\rightarrow K'$ is a morphism from $(K,\mu)$ to $(K',\mu')$ if and only if
   for every vertex $v$ of $K$,
   $$ \begin{cases}
   \mu'(\varphi(v)) \mid \mu(v),  &  \text{if $\mu$ and $\mu'$ are ascending}; \\
   \mu(v) \mid \mu'(\varphi(v)),  &  \text{if $\mu$ and $\mu'$ are descending}.
 \end{cases} $$    
 In particular, $\varphi$ is weight-preserving if and only if
 $\mu'(\varphi(v))=\mu(v)$ for every vertex $v$ of $K$.
   \end{lem}

The following notion was defined in~\cite[p.\,236]{Daw90} for ascending type weighted simplicial complexes. But  the parallel notion clearly makes sense for descending type weighted simplicial complexes as well.

  \begin{defi}[Contiguous Morphisms] \label{Def:Contigu}
  Suppose $(K,\mu)$ and $(K',\mu')$ are two weighted simplicial complexes of the same type. Two morphisms $$\varrho_0,\varrho_1: (K,\mu)\rightarrow (K',\mu')$$ are called \emph{contiguous}, denoted by $\varrho_0 \underset{c}{\simeq} \varrho_1$,  
  if there exists a morphism
  $$ F: (K \times [0, 1], \mu \times \mathbf{1}) \rightarrow (K',\mu')\  \text{with}\
  F(x, 0) =\varrho_0(x),\, F(x,1)= \varrho_1(x),\ \forall x\in K.$$
   Here $([0,1],\mathbf{1})$ is a weighted simplicial complex where the weight of every 
   simplex of $[0,1]$ is $1$, and $(K\times [0,1],\mu\times \mathbf{1})$ is the Cartesian product $(K,\mu)$ with
   $([0,1],\mathbf{1})$ (see Definition~\ref{Def:Product-Wt-Complexes}). 
   Moreover, if $\varrho_0$, $\varrho_1$ and $F$ are all weight-preserving, we say that 
   $\varrho_0$ and $\varrho_1$ are \emph{strongly contiguous}.
  \end{defi}

  \begin{thm}[\text{\cite[Theorem 2.5]{Daw90}}] \label{Thm:Contig}
   If two morphisms $\varrho_0,\varrho_1: (K,\mu)\rightarrow (K',\mu')$ of weighted simplicial complexes 
   are contiguous, then the chain maps $(\varrho_0)_{\#}$
   and $(\varrho_1)_{\#}$ are chain homotopic and hence  $ (\varrho_0)_*= (\varrho_1)_*:  H_*(K,\partial^{\mu})\rightarrow H_*(K',\partial^{\mu'})$. 
  \end{thm}

  \begin{prop}[\text{\cite[Theorem 3.1]{Daw90}}] \label{Prop:Wt-Div-Simplex}
  Let $\Delta^n$ be the standard $n$-simplex which is considered as a simplicial complex with respect to its facial structure. If $\mu$ is a divisible weight on $\Delta^n$, then
   $$
   H_j(\Delta^n,\partial^{\mu}) \cong \begin{cases}
   \Z ,  &  \text{if $j=0$}; \\
   0 ,  &  \text{if $j \geq 1$}.
 \end{cases}
 $$
  \end{prop}  
 Note that if a weight function $\mu$ on $\Delta^n$ is not divisible, it is well possible that $H_j(\Delta^n,\partial^{\mu}) \neq 0$ when $j >0$.\n
  
   The proofs of Theorem~\ref{Thm:Contig} and Proposition~\ref{Prop:Wt-Div-Simplex} in~\cite{Daw90} are for ascending type weighted simplicial complexes. 
   But the same proofs clearly work for descending
   type ones.
    The reader is referred to~\cite{Daw90} for more discussion of the categorical properties of weighted simplicial complexes.
 \n
 
 \subsection{Ascending type vs. descending type weighted simplicial complex}
 \ \n
 
 For a weighted simplicial complex $(K,\mu)$, let
 \begin{equation}\label{Equ:N-mu}
   N_{\mu} := \text{the least common multiple of all the weights $\mu(\sigma)$}, \ \sigma\in K. 
   \end{equation}
   
 \begin{lem} \label{Lem:Adjoint-Weight}
 For a weighted simplicial complex $(K,\mu)$ where $N_{\mu}$ is finite, there exists a  weighted simplicial complex $(K,\mu^*)$
 of opposite type such that
 \begin{equation} \label{Equ:Homology-Adjoint-Equal}
   H_i(K,\partial^{\mu^*}) \cong H_i(K,\partial^{\mu}) \ \text{for all}\ i \geq 0.
   \end{equation}
  \end{lem}
  \begin{proof}
 Define a weight $\mu^*$ of opposite type to $\mu$ on $K$ by 
\begin{equation} \label{Equ:Adjoint-Weight}
 \mu^*(\sigma) = \frac{N_{\mu}}{\mu(\sigma)}, \ \forall \sigma\in K.
 \end{equation}
 
 By the definitions of $\partial^{\mu}$ and $\partial^{\mu^*}$, $(C_*(K),\partial^{\mu})$ and $(C_*(K),\partial^{\mu^*})$ are in fact isomorphic chain complexes. So
 $H_*(K,\partial^{\mu^*}) \cong H_*(K,\partial^{\mu})$. 
  \end{proof}
 
  By Lemma~\ref{Lem:Adjoint-Weight}, we can consider descending type weighted simplicial complexes as the twin objects of ascending type ones. So many properties
 of ascending type weighted simplicial complexes proved in~\cite{Daw90} (such as Mayer-Vietoris sequence and excision) can be parallelly proved for descending type ones.  On the other hand,  not all properties of these two categories of objects are the same (see Section~\ref{Sec:Product-Cohomology} for more discussion).
 
\subsection{Notations of simplicial complexes}
 \ \n
 
   Let $K$ be a simplicial complex. We fix some notations for our discussions in the rest of the paper. 
 \begin{itemize}
  \item For any $n\geq 0$, denote by $K^{(n)}$ the $n$-skeleton of $K$. \n

   \item  Let $|K| \subseteq \R^N$ denote a \emph{geometric realization} of $K$ where
  each simplex $\sigma\in K$ determines a geometric simplex $|\sigma| \subseteq |K|$. So
   $ |K| = \bigcup_{\sigma\in K} |\sigma|$.\n

    \item For any simplex $\sigma$ of $K$, let $V(\sigma)$ denote the vertex set of $\sigma$ and let
      \begin{align*}
       \mathrm{star}_K(\sigma) &= \{ \tau\in K \,|\, \tau\cup\sigma\in K\} \\
         \mathrm{link}_K(\sigma) &= \{ \tau\in K \,|\, \tau\cup\sigma\in K, \tau\cap\sigma=\varnothing\}.
      \end{align*}
      
   \item  The \emph{open star} of a simplex $\sigma$ in $|K|$ is 
              $$ \mathrm{St}(\sigma,K) = |\mathrm{star}_K(\sigma)| \backslash |\mathrm{link}_K(\sigma)| \subseteq |K|.  $$ 
  \end{itemize}
 
\vskip 0.4cm

\section{Barycentric subdivision of Weighted Simplicial Complex} \label{Sec:Bary-Subdiv}

    For a simplicial complex $K$, let $Sd(K)$ denote the \emph{barycentric subdivision} of $K$ (see~\cite[\S 17]{Munk84}). By definition, $Sd(K)$ is a simplicial complex where any simplex in $Sd(K)$ is of the form $[b_{\sigma_0},\cdots, b_{\sigma_l}]$ where
   $\sigma_0\subsetneq \cdots\subsetneq \sigma_l \in K$
   and $b_{\sigma_i}$ is the barycenter of $\sigma_i$ for each $0\leq i \leq l$. We associate a sign
   $\varepsilon(\sigma_0,\cdots,\sigma_l) \in \{ 1, -1\}$
   to each $[b_{\sigma_0},\cdots, b_{\sigma_l}]$ as follows.   
   Suppose the vertices of $\sigma$ are ordered as $v_0,\cdots, v_n$. If 
  $  V(\sigma_s) = \{ v_{i_0},\cdots, v_{i_{k_s}} \}$ 
  with $i_0 < \cdots < i_{k_s}$ for each
  $ 0\leq s \leq l$, define
  \begin{equation} \label{Equ:Ver-Sequ}
    V(\sigma_0), V(\sigma_1)  - V(\sigma_{0}),\cdots,
   V(\sigma_l)  - V(\sigma_{l-1})  
   \end{equation}
   where the vertices in $V(\sigma_0)$ and each $V(\sigma_s) - V(\sigma_{s-1})$  are ordered from left to right according to the given order of $V(\sigma)$. So~\eqref{Equ:Ver-Sequ} determines a sequence of vertices  which is a permutation 
    of $(v_{i_0},\cdots, v_{i_{k_l}})$. Let
  $I(\sigma_0,\cdots,\sigma_l)$ denote
   the \emph{inversion number} of the sequence in~\eqref{Equ:Ver-Sequ}, i.e.
   $I(\sigma_0,\cdots,\sigma_l)$ counts the number of pairs of vertices $(v_{i}, v_{i'})$ in the sequence~\eqref{Equ:Ver-Sequ} with $i> i'$. Then define
   $$\varepsilon(\sigma_0,\cdots,\sigma_l) := (-1)^{I(\sigma_0,\cdots,\sigma_l)}.  $$
 
    Let $Sd_{\#} : C_*(K)\rightarrow C_*(Sd(K))$ denote the  chain map induced by $Sd$, which
  sends an oriented $n$-simplex $\sigma$ to the sum of all the $n$-simplices in $Sd(\sigma)$ with the same orientation. More specifically,
   \begin{equation} \label{Equ:Bary-Subdiv-Form}
   Sd_{\#}(\sigma)= \sum_{\sigma_0\subsetneq \cdots\subsetneq \sigma_n = \sigma} 
 \varepsilon(\sigma_0,\cdots,\sigma_n)\cdot [b_{\sigma_0}\cdots b_{\sigma_n}]. 
 \end{equation}
 We have $Sd_{\#}\circ\partial = 
  \partial\circ Sd_{\#}$. In addition, we can also define $Sd_{\#}$ inductively by
   \begin{equation} \label{Equ:Sd-simple}
    Sd_{\#}(\sigma)=b_{\sigma}\cdot Sd_{\#}(\partial \sigma),\ \text{and}\ Sd_{\#}(v)=v, \, \forall v\in K^{(0)},
    \end{equation}    
  where $b_{\sigma}$ is the barycenter of $\sigma$ and
  $b_{\sigma}\cdot \tau$ is the cone of $\tau$ over $b_{\sigma}$.    
 \n
   
    Moreover, for any $k$-face $\tau$ of an $n$-simplex $\sigma$, define 
   \begin{equation} \label{Equ:Sd-sigma-tau}
    Sd_{\#}(\sigma,\tau):= 
    \sum_{\tau=\sigma_k\subsetneq \cdots\subsetneq \sigma_{n}= \sigma} 
 \varepsilon(\sigma_k,\cdots,\sigma_{n})\cdot [b_{\sigma_k}\cdots b_{\sigma_{n}}].
 \end{equation}
 Indeed, $Sd_{\#}(\sigma,\tau)$ consists of all the $n-k$ simplices in $Sd(\sigma)$ which intersect $\tau$ at a single point $b_{\tau}$. This definition is useful in
 the proof of Theorem~\ref{Thm:Isom-Char}.\n
   
   We observe that for any weighted simplicial complex $(K,\mu)$, there is a natural divisible weight on $Sd(K)$ defined below.
     \n
     
 \begin{defi}[Barycentric Subdivision of Weighted Simplicial Complex] \label{Def:Bary-Subdiv-Weighted}
\ \n
 For a weighted simplicial complex $(K,\mu)$, define the weight of the barycenter $b_{\sigma}$ of any simplex $\sigma$ in $K$ to be $\mu(\sigma)$. Then define a weight  $Sd(\mu)$ on $Sd(K)$ by:
    \begin{equation} \label{Equ:Subdiv-w}
     Sd(\mu)\big([ b_{\sigma_0},\cdots, b_{\sigma_l}]\big)=\mu(\sigma_l),\ \sigma_0\subsetneq \cdots\subsetneq \sigma_l \in K.
     \end{equation}
   
   Note that we have either
   $\mu(\sigma_0)\, | \, \cdots \, | \,\mu(\sigma_l)$ or 
    $\mu(\sigma_l) \, |\, \cdots\, |\, \mu(\sigma_0)$ depending on whether $\mu$ is ascending or descending. Then it is easy to 
    check that $Sd(\mu)$ is a divisible weight on $Sd(K)$ which is ascending (or descending) if so is $\mu$.  
     We call $(Sd(K),Sd(\mu))$ the \emph{barycentric subdivision} of $(K,\mu)$. For any integer $m\geq 1$, let $(Sd^m(K),Sd^m(\mu))$ 
 denote the $m$ iterated barycentric subdivisions of $(K,\mu)$.
  \end{defi}

If $\sigma$ is an $n$-simplex of $K$, every $n$-simplex $\tau$ in $Sd(\sigma)$ is of the form $[b_{\sigma_0},\cdots, b_{\sigma_n}]$ where
$\sigma_0\subsetneq \cdots\subsetneq \sigma_n=\sigma$.
Then by the definition of $Sd(\mu)$, we have
\begin{equation} \label{Equ-Sd-n-simplex}
 Sd(\mu)(\tau) = \mu(\sigma).
 \end{equation}
 So in~\eqref{Equ:Bary-Subdiv-Form}, $Sd_{\#}(\sigma)$ consists of $n$-simplices with the same weight as $\sigma$.

\begin{lem} \label{Lem:Barycen-Sub-Iso-Weight}
 For any weighted simplicial complex $(K,\mu)$, the
 weighted spaces $(|K|,\lambda_{\mu})$ and
 $\big(|Sd(K)|, \lambda_{Sd(\mu)}\big)$ are isomorphic.
\end{lem}
\begin{proof}
 For any simplex $\sigma$ of $K$, we can write
 the relative interior $|\sigma|^{\circ}$ of $|\sigma|$ as
 $$ |\sigma|^{\circ} =  \underset{\sigma_0\subsetneq \cdots\subsetneq \sigma_l =\sigma}{\bigcup_{[b_{\sigma_0}\cdots b_{\sigma_l}]\in Sd(K)}} |[b_{\sigma_0}\cdots b_{\sigma_l}]|^{\circ}. $$
 
 For any point $x\in |\sigma|^{\circ}$, it follows from Definition~\ref{Def:Realization-WSC} of $\lambda_{\mu}$ that
$$\lambda_{\mu}(x) =  \mu(\sigma) = \lambda_{\mu}(b_{\sigma}). $$

 Meanwhile, $x$ must lie in the relative interior 
 of some simplex $[ b_{\sigma_0}\cdots b_{\sigma_l}]$
 of $Sd(K)$
 with $\sigma_0\subsetneq \cdots\subsetneq \sigma_l =\sigma$. So by the definition of $Sd(\mu)$ (see~\eqref{Equ:Subdiv-w}), 
 $$  \lambda_{Sd(\mu)} (x) = Sd(\mu)\big([b_{\sigma_0}\cdots b_{\sigma_l}]\big) = \mu(\sigma_l) = \mu(\sigma). $$ 
 So the two weighted spaces $(|K|,\lambda_{\mu})$ and
 $\big(|Sd(K)|, \lambda_{Sd(\mu)}\big)$ are isomorphic.
\end{proof}

By Definition~\ref{Def:Bary-Subdiv-Weighted}, the barycentric subdivision of a weighted simplicial complex
 is always a divisibly weighted simplicial complex. So we obtain the follow corollary immediately from Lemma~\ref{Lem:Barycen-Sub-Iso-Weight}.
 
\begin{cor}\label{Cor:Div-Weight-Triangl}
 Any weighted polyhedron
   has a divisibly weighted triangulation. 
\end{cor}
 \n

 \begin{lem} \label{Lem:Sd-Weight-Chain}
   $Sd_{\#}: \big(C_*(K),\partial^{\mu} \big) \rightarrow 
   \big(C_*(Sd(K)),\partial^{Sd(\mu)} \big)$ is a chain map. 
  \end{lem}
  \begin{proof}
  For an oriented $n$-simplex $\sigma$ of $K$, 
     \begin{align*}
    \partial^{Sd(\mu)}\circ Sd_{\#}(\sigma) 
   & \overset{\eqref{Equ:Sd-simple}}{=} \partial^{Sd(\mu)}\big( b_{\sigma}\cdot 
    Sd_{\#}(\partial \sigma ) \big)
    = \partial^{Sd(\mu)}\Big( b_{\sigma}\cdot 
    \sum^n_{j=0} (-1)^j   Sd_{\#}(\partial_j\sigma)  \Big) \\
    &=   \sum^n_{j=0} (-1)^j  \partial^{Sd(\mu)} \big( b_{\sigma}\cdot  Sd_{\#}(\partial_j\sigma)  \big) \\
    &\overset{\divideontimes}{=} 
    \sum^n_{j=0} (-1)^j \Big( c^{\mu}_j(\sigma) Sd_{\#}(\partial_j\sigma) - b_{\sigma}\cdot \partial \,
    Sd_{\#}(\partial_j \sigma )  \Big) \\ 
     (\text{by induction})  &= Sd_{\#} \Big( \sum^n_{j=0} (-1)^j c^{\mu}_j(\sigma) \partial_j\sigma \Big) - b_{\sigma}\cdot  Sd_{\#}(\partial\partial\sigma ) \overset{\eqref{Equ:Boundary-Unified}}{=}
    Sd_{\#}\circ \partial^{\mu} \sigma.
   \end{align*}
  The equality $\overset{\divideontimes}{=}$ follows from the facts that for any $(n-1)$-simplex $\tau$ in $Sd(\partial_j \sigma)$:\n
  
  $\bullet$ $Sd(\mu)(\tau)= \mu(\partial_j\sigma)$ by~\eqref{Equ-Sd-n-simplex} and $Sd(\mu)(b_{\sigma}\cdot \tau)=\mu(\sigma)$ by~\eqref{Equ:Subdiv-w}, which together give the coefficient $c^{\mu}_j(\sigma)$ defined in~\eqref{Equ:Boundary-Unified};\n
  
 $\bullet$ for each $(n-2)$-face $\kappa$ of $\tau$, $Sd(\mu)(b_{\sigma}\cdot \kappa) =\mu(\sigma)$ by~\eqref{Equ:Subdiv-w}.
 \end{proof}  
  
  \vskip 0.4cm

\section{Weighted singular homology of weighted polyhedron} \label{Sec:Weigh-Singul-Homology}

 In this section, we first define some basic notions on weighted polyhedra. Then  we introduce the weighted singular homology of a weighted polyhedron which is a natural generalization of singular homology to weighted spaces. 
 Moreover, we will prove that the weighted singular homology
 of a weighted polyhedron $(X,\lambda)$ is
 isomorphic to the weighted simplicial homology of 
 any divisibly weighted triangulation of $(X,\lambda)$.
 \n
 
 \subsection{Basic definitions on weighted polyhedron}
 \ \n

\begin{defi}[W-continuous map] \label{Def:W-cont-map}
   Suppose $(X,\lambda)$ and $(X',\lambda')$ are
   weighted polyhedra of the same type. 
   A \emph{W-continuous map} from $(X,\lambda)$ to 
   $(X',\lambda')$ is a continuous map $f: X\rightarrow X'$   which satisfies: for all $x\in X$,
   $$ \begin{cases}
   \lambda'(f(x)) \mid \lambda(x),  &  \text{if  $(X,\lambda)$ and $(X',\lambda')$ are of ascending type}; \\
   \lambda(x) \mid \lambda'(f(x)),  &  \text{if  $(X,\lambda)$ and $(X',\lambda')$ are of descending type}.
 \end{cases}  
   $$  
    We also denote such a map by $f: (X,\lambda)\rightarrow (X',\lambda')$ to indicate its W-continuity
    since the definition of $f$ depends on the type of the weighted polyhedra involved.
   Moreover, $f$ is called \emph{weight-preserving} if 
$\lambda'(f(x))=\lambda(x)$ for all $x\in X$.\n

We say that $(X,\lambda)$ is \emph{isomorphic} to $(X',\lambda')$ if there exists a weight-preserving homeomorphism from $(X,\lambda)$ to $(X',\lambda')$.  
 \end{defi}

  \begin{exam}[Geometric Realization of a Morphism] \label{Exam:Geo-Realiztion}
    If $\varrho: (K,\mu)\rightarrow (K',\mu')$ is a
   morphism between two weighted simplicial complexes of the same type, then the induced map on the weighted polyhedra, denoted by $|\varrho|: (|K|,\lambda_{\mu})\rightarrow
 (|K'|,\lambda_{\mu'})$, is W-continuous. We call $|\varrho|$ the \emph{geometric realization} of $\varrho$.  
 In particular, $|\varrho|$ is weight-preserving if so is $\varrho$. 
 \end{exam}
 
  \begin{itemize}
   \item[\textbf{(CV-5)}] Suppose $(K,\mu)$ is a weighted simplicial complex and $(X,\lambda)$ is a weighted space.
    If $f: (|K|,\lambda_{\mu}) \rightarrow (X,\lambda)$ is a W-continuous map, for any simplex $\sigma$ of $K$, denote the restriction of $f$ to $|\sigma|$ by:
    $$f|_{\sigma}: (|\sigma|,\lambda_{\mu|_{\sigma}}) \rightarrow (X,\lambda).
     $$
   \end{itemize}
 
  \begin{defi}[Stratum]  \label{Def:Stratum}
      For a weighted polyhedron $(X,\lambda)$, let 
     $$O^n(X,\lambda):=\{x\in X\,|\, \lambda(x)=n\} \subseteq X,\ n\in \mathrm{Im}(\lambda).$$
      We call $O^n(X,\lambda)$ a \emph{stratum} of $(X,\lambda)$.
So $X$ is the disjoint union of all its strata:
 $$X = \bigcup_{n\in \mathrm{Im}(\lambda)} O^n(X,\lambda).$$
  \end{defi}

 The \emph{Cartesian product of two weighted spaces} 
 $(X,\lambda)$ and $(X',\lambda')$ is a weighted space
 $(X\times X', \lambda\times \lambda')$ where
 $(\lambda\times \lambda')(x,x') = \lambda(x)\cdot\lambda'(x')$
 for all $x\in X$, $x'\in X'$.
 
\begin{lem} \label{Lem:Product-Weighted-Poly}
Suppose
 $(X,\lambda)$ and $(X',\lambda')$ are two weighted polyhedra of the same type, then their Cartesian product
 $(X\times X', \lambda\times \lambda')$ is also
 a weighted polyhedron with the same type as
 $(X,\lambda)$ and $(X',\lambda')$.
 \end{lem}
 \begin{proof}
 Let $(K,\mu)$ and $(K',\mu')$ be weighted triangulations of $(X,\lambda)$ and  $(X',\lambda')$, respectively.
 We claim that the Cartesian product $(K\times K',\mu\times \mu')$
 is a weighted triangulation of $(X\times X', \lambda\times \lambda')$. Indeed,
 suppose a point $(x,x')\in K\times K'$ is carried by
 a simplex $
   \{(v_{i_1},v'_{j_1}),\cdots, (v_{i_s},v'_{j_s})\}$ of $K\times K'$
  where
  $\{v_{i_1},\cdots, v_{i_s}\}$ is a simplex in $K$
  and $\{v'_{j_1},\cdots, v'_{j_s}\}$ is a simplex in $K'$.
  Then $x$ is carried by $\{v_{i_1},\cdots, v_{i_s}\}$
  while $x'$ is carried by $\{v'_{j_1},\cdots, v'_{j_s}\}$.
  So by Definition~\ref{Def:Product-Wt-Complexes}, 
\begin{align*}
  \mu\times\mu'\big(\{(v_{i_1},v'_{j_1}),\cdots, (v_{i_s},v'_{j_s})\}\big) &\overset{\eqref{Equ:Product-Weight}}{=} \mu(\{v_{i_1},\cdots, v_{i_s}\})\cdot \mu'(\{v'_{j_1},\cdots, v'_{j_s}\}) \\
   &\overset{\eqref{Equ:lambda-mu-relation}}{=} \lambda(x)\lambda'(x') = \lambda\times\lambda'((x,x')).
\end{align*}  
 Clearly, $(K\times K',\mu\times \mu')$ has the same type as $(K,\mu)$ and $(K',\mu')$ by definition.
   \end{proof} 
  
    \begin{defi}[W-Homotopy]\label{Def:W-homotopy-Equiv} Let $(X,\lambda)$ and $(X',\lambda')$ be 
    weighted polyhedra of the same type and
       $f, g: (X,\lambda)\rightarrow (X',\lambda')$ be two  
     W-continuous maps. A \emph{W-homotopy} from $f$ to $g$ relative
to $A\subseteq X$ is a W-continuous map
$$H : (X \times [0, 1],\lambda\times \mathbf{1}) \rightarrow (X',\lambda')$$ which satisfies:
 \begin{itemize}
   \item $H(x, 0) =f(x)$ and $H(x,1)=g(x)$ for all $x\in X$;\n
   \item $H(a,t)=f(a)=g(a)$ for all $a\in A$ and 
   $t\in [0,1]$.\n
   \end{itemize}
   
   Here $([0,1],\mathbf{1})$ is considered as a weighted polyhedron, and $(X\times [0,1],\lambda\times \mathbf{1})$ is the product $(X,\lambda)$ with
   $([0,1],\mathbf{1})$.
   \n
   If there exists a W-homotopy from $f$ to $g$,
   we say that $f$ is \emph{W-homotopic to} $g$.
   If moreover $H$ is weight-preserving, we say
   that $f$ is \emph{strongly W-homotopic to} $g$ (a priori $f$ and $g$ must be both weight-preserving).
    \end{defi}
    
    The following lemma is obvious from
    the above definition.
     
     \begin{lem}
    If two morphisms $\varrho_0,\varrho_1: (K,\mu)\rightarrow (K',\mu')$ of weighted simplicial complexes are (strongly) contiguous, then the maps $\overline{\varrho}_0,
     \overline{\varrho}_1 : (|K|,\lambda_{\mu})\rightarrow
 (|K'|,\lambda_{\mu'})$ are (strongly) W-homotopic.
 \end{lem}
    
   \begin{rem}
    ``W-continuous map'' and ``W-homotopy''
   have been defined by Takeuchi and Yokoyama~\cite[Sec.\,5]{TakYok07} for orbifolds which require all the maps to be weight-preserving. Here we redefine these two terms for weighted polyhedra.
    \end{rem}
    
  \begin{defi}[W-Homotopy Equivalence] \label{Def:W-Homo-Equiv} 
  Suppose $(X,\lambda)$ and $(X',\lambda')$ are two weighted polyhedra of the same type. We say that $(X,\lambda)$ is \emph{W-homotopy equivalent} to 
  $(X',\lambda')$ if there exist W-continuous maps
  $$f: (X,\lambda)\rightarrow (X',\lambda'), \ \ 
  g : (X',\lambda')\rightarrow (X,\lambda)$$
such that $g \circ f$ and $f \circ g$
 are W-homotopic to $\mathrm{id}_{X}$ and 
 $\mathrm{id}_{X'}$, respectively. The map $f$ is called a
 \emph{W-homotopy equivalence} from $(X,\lambda)$ to $(X',\lambda')$, and $g$ is called a \emph{W-homotopy inverse} of $f$.  \n
 Moreover,
  $(X,\lambda)$ and $(X',\lambda')$ are called 
  \emph{strongly W-homotopy equivalent} if in the above definition,  $f$ and $g$ are both weight-preserving, and $g \circ f$ and $f \circ g$
 are strongly W-homotopic to $\mathrm{id}_{X}$ and 
 $\mathrm{id}_{X'}$, respectively. If so, we call $f$ and $g$ \emph{strong W-homotopy equivalences}.
\end{defi}

 \begin{exam}
  A weighted polyhedron $(X,\lambda)$ is always strongly W-homotopy equivalent to 
  $(X\times [0,1],\lambda\times \mathbf{1})$. 
 \end{exam}

 \begin{rem} \label{Rem:Quotient-Not}
   It is not meaningful to define the quotient space of a weighted polyhedron in general. This is because if $(X,\lambda)$ is a weighted polyhedron and $A$ is a subspace of $X$, the restriction of $\lambda$ to $A$ may not be constant. So there is no natural weight on the quotient space $X\slash A$ induced from $\lambda$.
 \end{rem}

 \subsection{Weighted singular homology}
 \ \n
 
 Let $(X,\lambda)$ be a weighted polyhedron.
 A \emph{weighted singular $n$-simplex} of $(X,\lambda)$ is a W-continuous map
 $\theta: (|\Delta^n|,\lambda_{\xi}) \rightarrow (X,\lambda)$ where $(\Delta^n,\xi)$ is a
  divisibly weighted $n$-simplex. If $\sigma$ is a $k$-face of $\Delta^n$, then 
 $\theta|_{\sigma}: (|\sigma|, \lambda_{\xi|_{\sigma}}) \rightarrow (X,\lambda)$ 
is a weighted singular $k$-simplex of $(X,\lambda)$. \n
 
 Let $\mathcal{W}_n(X,\lambda)$ denote the free abelian group
 generated by all weighted singular $n$-simplices of $(X,\lambda)$.  The boundary map
  on $\mathcal{W}_n(X,\lambda)$ is
 $$ \partial^{\lambda}: \mathcal{W}_n(X,\lambda) \rightarrow 
 \mathcal{W}_{n-1}(X,\lambda)$$
 where for each weighted singular simplex $\theta: (|\Delta^n|,\lambda_{\xi}) \rightarrow (X,\lambda)$,
  \begin{equation} \label{Equ:Bd-Sing}
   \partial^{\lambda} \theta:=
   \sum^n_{j=0}  (-1)^j c^{\xi}_j(\Delta^n)\cdot \theta|_{\partial_j \Delta^n}. 
   \end{equation}
  Here the coefficient $c^{\xi}_j(\Delta^n)$ is defined in~\eqref{Equ:Boundary-Unified} and $\partial_j \Delta^{n}$ is canonically identified with $\Delta^{n-1}$. It is clear that $\partial^{\lambda}\circ \partial^{\lambda}=0$.  We call
  $\big(\mathcal{W}_*(X,\lambda), \partial^{\lambda} \big)$
 the \emph{weighted singular chain complex} of $(X,\lambda)$, and call its homology $H_*\big(\mathcal{W}_*(X,\lambda), \partial^{\lambda} \big)$ the \emph{weighted
 singular homology} of $(X,\lambda)$, denoted by
 $\mathcal{H}^{W}_*(X,\lambda)$. 
 \n
 
If $f: (X,\lambda) \rightarrow (X',\lambda')$ is a W-continuous map, then $f$ induces a linear map
\begin{align} \label{Equ:f-chain-1}
  f_{\#}: \big(\mathcal{W}_*(X,\lambda), \partial^{\lambda} \big) &\longrightarrow \big(\mathcal{W}_*(X',\lambda'), \partial^{\lambda'} \big) \\
 \theta \ \ \ &\longmapsto\ \ \  f\circ\theta \notag
   \end{align}
   for each weighted singular simplex $\theta$ of $(X,\lambda)$.   
  It is clear that $f_{\#}$ is a
   chain map, and hence induces a homomorphism 
  $$ f_*: \mathcal{H}^{W}_*(X,\lambda) \rightarrow \mathcal{H}^{W}_*(X',\lambda').$$
  
 By the above definitions, it is easy to see that weighted singular homology is an invariant under isomorphisms of weighted polyhedra. Moreover, the following proposition shows that
  it is an invariant under W-homotopy equivalences, which generalizes the homotopy invariance of the ordinary singular homology.\n
  
  \begin{prop}
 Let $f, g: (X,\lambda)\rightarrow (X',\lambda')$ be
  two W-continuous maps and $f$ is W-homotopic to $g$.
  Then $f_*=g_*: \mathcal{H}^{W}_*(X,\lambda) \rightarrow \mathcal{H}^{W}_*(X',\lambda')$.
  In particular, any W-homotopy equivalence between $(X,\lambda)$ and $(X',\lambda')$ induces an isomorphism on weighted singular homology.
   \end{prop}
   \begin{proof}
   For the standard simplex $\Delta^n$, 
  in $\Delta^n \times [0,1]$, let 
  $$\Delta^n \times\{0\}=\left[v_0, \cdots, v_n\right], \ \ \Delta^n \times\{1\}=\left[w_0, \cdots, w_n\right],$$
   where $v_i$ and $w_i$ have the same image under the projection $\Delta^n \times [0,1] \rightarrow \Delta^n$. So $\Delta^n \times [0,1]$ is the union of the $(n+1)$-simplices $\left[v_0, \cdots, v_i, w_i, \cdots, w_n\right]$.\n
  Let $H$ be a W-homotopy from $f$ to $g$. 
  Then for any weighted singular $n$-simplex $\theta: (|\Delta^n|,\lambda_{\xi}) \rightarrow (X,\lambda)$, we have a W-continuous map
  $$H \circ(\theta \times id_{[0,1]}): (|\Delta^n|\times [0,1],\lambda_{\xi}\times\mathbf{1}) \rightarrow (X \times [0,1], \lambda\times \mathbf{1}) \rightarrow (X',\lambda').$$
    Define a family of linear maps 
    $\{P_n: \mathcal{W}_n(X,\lambda) \rightarrow \mathcal{W}_{n+1}(X',\lambda') \}_{n\geq 0}$ by:
  \begin{align*}
      P_n( \theta )
   :=\sum^{n}_{i=0}(-1)^i
  H \circ(\theta \times id_{[0,1]}) \mid_{\left[v_0, \cdots, v_i, w_i, \cdots, w_n\right]} . 
  \end{align*}
     It is routine to check that $\partial^{\lambda'} P_n =g_{\#}-f_{\#}-P_{n-1} \partial^{\lambda}$ for all $n\geq 0$.
  This implies that $f_*$ agrees with $g_*$ on weighted singular homology. 
\end{proof}

   \subsection{Isomorphism between weighted singular homology and weighted simplicial homology} \label{Subsec:Isomorphism}
 \ \n
   
  The isomorphism between the ordinary singular homology and the simplicial homology of a simplicial complex $K$ can be proved by the excision theorem (see Hatcher~\cite[Theorem 2.27]{Hatcher02}).
  But the proof involves the computation of
  the singular homology groups of the quotient space $K^{(n)} \slash K^{(n-1)}$, which cannot be generalized to  weighted polyhedra (see Remark~\ref{Rem:Quotient-Not}).
  In the following, we will prove the isomorphism using the
  generalized Mayer-Vietoris sequence (see Bott and Tu~\cite{BotTu82}) and the \v{C}ech complex of a sheaf. \n
  
  Let $(X,\lambda)$ be a weighted polyhedron and $(K,\mu)$
  be an arbitrary divisibly weighted triangulation of $(X,\lambda)$. We identify $X$ with the geometric realization
  $|K|$ of $K$.
  The open stars of all the vertices of $K$ form an open cover of $|K|$, denoted by $\mathcal{U}$, that is
  \begin{equation} \label{Equ:U-Cover}
    \mathcal{U}= \{ \mathrm{St}(v,K), v\in K^{(0)} \}. 
    \end{equation}
   
  For any simplex $\sigma$ of $K$, we have
  $$ \mathrm{St}(\sigma, K) = \bigcap_{v\in V(\sigma)}
  \mathrm{St}(v,K). $$ 
    Clearly,
  $\big(\mathrm{St}(\sigma, K), \lambda|_{\mathrm{St}(\sigma, K)}\big)$ is also a weighted polyhedron. For brevity, we denote $\big(\mathrm{St}(\sigma, K), \lambda|_{\mathrm{St}(\sigma, K)}\big)$ by 
  $(\mathrm{St}(\sigma, K), \lambda)$.
     Then we obtain a sequence of inclusions
\[\begin{tikzcd}
	X=|K| & {\coprod\limits_{\mathrm{dim}(\sigma)=0} \mathrm{St}(\sigma,K)} & {\coprod\limits_{\mathrm{dim}(\sigma)=1} \mathrm{St}(\sigma,K)} & \cdots
	\arrow[from=1-2, to=1-1]
	\arrow[shift right, from=1-3, to=1-2]
	\arrow[shift left, from=1-3, to=1-2]
	\arrow[from=1-4, to=1-3]
	\arrow[shift left=3, from=1-4, to=1-3]
	\arrow[shift right=3, from=1-4, to=1-3]
\end{tikzcd}\]

Let $\mathcal{W}_*^{\mathcal{U}}(X,\lambda)$
 denote the group of \emph{$\mathcal{U}$-small
chains} in $\mathcal{W}_*(X,\lambda)$. These are chains made up of weighted singular simplices
 $\theta: (|\Delta^n|,\lambda_{\xi}) \rightarrow (X,\lambda)$ 
where the image of $\theta: |\Delta^n|\rightarrow X$ lies in some open set
$\mathrm{St}(v,K)$ in $\mathcal{U}$.

\begin{lem} \label{Lem:Inclusion-Equiv}
The inclusion
$(\mathcal{W}_*^{\mathcal{U}}(X,\lambda),\partial^{\lambda}) \hookrightarrow (
\mathcal{W}_*(X,\lambda),\partial^{\lambda})$
is a chain homotopy equivalence.
\end{lem}

 The proof of this lemma is completely parallel to the proof for the ordinary singular homology in~\cite[Proposition 2.21]{Hatcher02}.
 The idea is quite intuitive: to get an inverse chain map,
 we can use barycentric subdivision to subdivide each chain in $\mathcal{W}_*(X,\lambda)$ until it becomes $\mathcal{U}$-small. Since the proof is a little technical and lengthy, we put it in the appendix at the end.\n

 Define the \v{C}ech boundary operator
 $$d: \bigoplus_{\mathrm{dim}(\sigma)=n} \mathcal{W}_q\left(\mathrm{St}(\sigma,K) \right) \longrightarrow \bigoplus_{\mathrm{dim}(\sigma)=n-1} \mathcal{W}_q\left(\mathrm{St}(\sigma,K) \right) $$
 by the ``alternating sum formula'': for any weighted singular $q$-simplex $\theta \in \mathcal{W}_q\left(\mathrm{St}(\sigma,K)\right)$,
 $$ d \theta := \Big( (-1)^j  (\iota_{j})_{\#} (\theta) \Big)_{0\leq j \leq n} \in \bigoplus^n_{j=0} \mathcal{W}_q\left(\mathrm{St}(\partial_j\sigma,K) \right),$$
 where $\iota_j :  \mathrm{St}(\sigma,K) \hookrightarrow 
 \mathrm{St}(\partial_j\sigma,K)$ is the inclusion.
 It is routine to check  
 $d\circ d=0$.

 \begin{lem}[Generalized Mayer-Vietoris sequence for weighted singular chains] \label{Lem:MV-Seq}

 Let $(K,\mu)$ be a divisibly weighted triangulation of $(X,\lambda)$. Then there is an exact sequence of chain complex
$$
0 \leftarrow \mathcal{W}_q^{\mathcal{U}}(X,\lambda) \stackrel{\varepsilon}{\longleftarrow} \underset{\mathrm{dim}(\sigma)=0}{\bigoplus} \mathcal{W}_q\left( \mathrm{St}(\sigma,K),\lambda\right) \stackrel{d}{\longleftarrow} \underset{\mathrm{dim}(\sigma)=1}{\bigoplus} \mathcal{W}_q\left(\mathrm{St}(\sigma,K),\lambda\right) \stackrel{d}{\longleftarrow} \cdots$$
where $\varepsilon$ is summing up the chains from all
the open stars via inclusions into $X$.
\end{lem}
\begin{proof}
  Using the sheaf theoretic language, $(\mathcal{W}_*^{\mathcal{U}}(X,\lambda),\partial^{\lambda})$ is a differential cosheaf on $\mathcal{U}$  (see Bredon~\cite[Chapter VI]{Bredon97}). So by~\cite[p.\,426, Theorem 4.4]{Bredon97},
  the above sequence is exact. We can also prove this lemma using the similar argument as~\cite[Proposition 15.2]{BotTu82}. 
\end{proof}

\begin{thm}\label{Thm:Main-Isom}
For any
divisibly weighted triangulation $(K,\mu)$ of 
a weighted polyhedron $(X,\lambda)$, there is an isomorphism between
$H_*(K,\partial^{\mu})$ and $\mathcal{H}^{W}_*(X,\lambda)$.
\end{thm}
\begin{proof}
  Consider the following augmented double complex:  
\[
% 关键修改：添加 scale=0.8 和 every node/.style={scale=0.8}
% 缩放比例自己改：0.7=70%，0.8=80%，0.9=90%
\begin{tikzpicture}[
    baseline=(diag.center), 
    column sep= small , 
    row sep = tiny,
    scale=0.85,            % 整体缩放图形
    every node/.style={scale=0.85} % 同步缩放所有文字/公式
]
% 1. draw diagram
\node (diag) at (0,0) {
$\begin{tikzcd}
	{0 } & {\mathcal{W}_2^{\mathcal{U}}(X,\lambda) } & {\underset{\dim(\sigma)=0}{\bigoplus}\, \mathcal{W}_2\left( \mathrm{St}(\sigma,K),\lambda\right)} && \\
	{0 } & {\mathcal{W}_1^{\mathcal{U}}(X,\lambda) } & {\underset{\dim(\sigma)=0}{\bigoplus}\, \mathcal{W}_1\left( \mathrm{St}(\sigma,K),\lambda\right)} \\
	{0 } & {\mathcal{W}_0^{\mathcal{U}}(X,\lambda)} & {\underset{\dim(\sigma)=0}{\bigoplus}\, \mathcal{W}_0\left( \mathrm{St}(\sigma,K),\lambda\right)} & {\underset{\dim(\sigma)=1}{\bigoplus} \mathcal{W}_0\left(\mathrm{St}(\sigma,K), \lambda \right)} & {\underset{\dim(\sigma)=2}{\bigoplus} \mathcal{W}_0\left(\mathrm{St}(\sigma,K),\lambda \right)} \\
	&& {C_0(K) } & {C_1(K) } & {C_2(K)} \\
	&& {0 } & {0 } & {0 }
	\arrow[from=1-2, to=1-1]
	\arrow["\scalebox{1.4}{$\varepsilon$}"',from=1-3, to=1-2]
	\arrow[from=2-2, to=2-1]
	\arrow["\scalebox{1.4}{$\varepsilon$}"',from=2-3, to=2-2]
	\arrow[from=3-2, to=3-1]
	\arrow["\scalebox{1.4}{$\varepsilon$}"',from=3-3, to=3-2]
	\arrow["\scalebox{1.4}{$\eta$}"',from=3-3, to=4-3]
	\arrow["\scalebox{1.4}{$\eta$}"',from=3-4, to=4-4]
	\arrow["\scalebox{1.4}{$\eta$}"',from=3-5, to=4-5]
	\arrow[from=4-3, to=5-3]
	\arrow[from=4-4, to=5-4]
	\arrow[from=4-5, to=5-5]
	\arrow["\scalebox{1.4}{$\partial^{\lambda}$}"',from=1-2, to=2-2]
	\arrow["\scalebox{1.4}{$\partial^{\lambda}$}"',from=2-2, to=3-2]
	\arrow["\scalebox{1.4}{$\partial^{\lambda}$}"',from=1-3, to=2-3]
	\arrow["\scalebox{1.4}{$\partial^{\lambda}$}"',from=2-3, to=3-3]
	\arrow["\scalebox{1.4}{$d$}"',from=3-4, to=3-3]
	\arrow["\scalebox{1.4}{$d$}"',from=3-5, to=3-4]
	\arrow["\scalebox{1.4}{$\partial^{\mu}$}"',from=4-4, to=4-3]
	\arrow["\scalebox{1.4}{$\partial^{\mu}$}"',from=4-5, to=4-4]
\end{tikzcd}$
};

% 2. darw lines
\draw[-{Latex[length=1.8mm]}, thick] (-5.04, -0.81) -- (-5.04, 4.0) node[left] {$q$};
\draw[-{Latex[length=1.8mm]}, thick] (-5.04, -0.81) -- (9.26, -0.81) node[below right] {$p$};
\draw[thick] (-0.04, -0.81) -- (-0.04, 3.5);
\draw[thick] (4.9, -0.81) -- (4.9, 3.5);

\end{tikzpicture}
\]
where for each $n$-simplex $\sigma$ of $K$, the map $\eta: \mathcal{W}_0\left( \mathrm{St}(\sigma,K),\lambda\right) \rightarrow C_n(K)$ is defined by: for a weighted singular $0$-simplex $\theta_0: (|\Delta^0|,\lambda_{\xi})
 \rightarrow \left(\mathrm{St}(\sigma,K),\lambda  \right)$,
  \begin{equation} \label{Equ:Def-Eta}
   \eta(\theta_0) = \begin{cases}
\displaystyle  \frac{\xi(\Delta^0)}{\mu(\sigma)}\sigma,  &  \text{if $(X,\lambda)$ is of ascending type}; \\
 \displaystyle  \overset{\ \ }{ \frac{\mu(\sigma)}{\xi(\Delta^0)}}\sigma,  &  \text{if $(X,\lambda)$ is of descending type}.
 \end{cases} 
 \end{equation}
 
 It is easy to check that $\eta\circ d = \partial^{\mu}\circ \eta$. \n
 
\noindent \textbf{Claim:} The coefficient of $\sigma$ in $\eta(\theta_0)$ above is always an integer.\n
 
 Let $x_0 = \theta_0(\Delta^0)\in \mathrm{St}(\sigma,K)$. Then since $\theta_0$ is W-continuous, we have
 \begin{equation} \label{Equ:divide-1}
   \begin{cases}
  \lambda(x_0) \mid \lambda_{\xi}(\Delta^0),  &  \text{if $(X,\lambda)$ is of ascending type}; \\
  \lambda_{\xi}(\Delta^0) \mid   \lambda(x_0),  &  \text{if $(X,\lambda)$ is of descending type}.
 \end{cases}
 \end{equation}
 On the other hand, for any simplex 
  $\sigma'$ of $K$ that contains $\sigma$,  
  $$   \begin{cases}
  \mu(\sigma) \mid \mu(\sigma'),  &  \text{if $\mu$ is ascending}; \\
    \mu(\sigma') \mid \mu(\sigma),  &  \text{if $\mu$ is descending}.
 \end{cases} 
  $$
 Then since $\sigma\subseteq \mathrm{Car}_K(x_0)$, $\lambda(x_0) = \lambda_{\mu}(x_0)= \mu(\mathrm{Car}_K(x_0))$ satisfies:
  \begin{equation} \label{Equ:divide-2}
     \begin{cases}
  \mu(\sigma)\mid \lambda(x_0),  &  \text{if $(X,\lambda)$ is of ascending type}; \\
 \lambda(x_0) \mid \mu(\sigma),  &  \text{if $(X,\lambda)$ is of descending type}.
 \end{cases}
 \end{equation}
  The above claim follows from~\eqref{Equ:divide-1} and~\eqref{Equ:divide-2} immediately.\n

  By Lemma~\ref{Lem:MV-Seq}, every row in the above diagram is exact. So
the total homology of the double complex is isomorphic to the homology of the left column, which is isomorphic to the homology of $(
\mathcal{W}_*(X,\lambda),\partial^{\lambda})$ by Lemma~\ref{Lem:Inclusion-Equiv}. Moreover, Lemma~\ref{Lem:Retract-Barycenter} below tells us that every column of the double complex is exact as well. Then
the total homology of the double complex is isomorphic to the homology of the bottom row, which is exactly the weighted simplicial homology $H_*(K,\partial^{\mu})$.
\end{proof}

\begin{lem} \label{Lem:Retract-Barycenter}
 Let $(K,\mu)$ be a divisibly weighted simplicial complex.
  For any $k$-simplex $\sigma$ of $K$, the following chain complex is exact for $\mathrm{St}(\sigma,K)\subseteq |K|$:
   $$  \cdots  \overset{\partial^{\lambda_{\mu}}}{\longrightarrow}  \mathcal{W}_1\left(\mathrm{St}(\sigma,K),\lambda_{\mu} \right) \overset{\partial^{\lambda_{\mu}}}{\longrightarrow} \mathcal{W}_0\left(\mathrm{St}(\sigma,K),\lambda_{\mu} \right)  \overset{\eta}{\longrightarrow}  C_k(K)\longrightarrow 0. $$
   \end{lem}
\begin{proof}
For brevity, let $\lambda=\lambda_{\mu}$. First, 
we claim that $(\mathrm{St}(\sigma,K), \lambda)$ is W-homotopy equivalent to the weighted $0$-simplex $(b_{\sigma}, \lambda|_{b_{\sigma}})$, where $b_{\sigma}$ is the barycenter of $\sigma$. This is because
for any point $x\in \mathrm{St}(\sigma,K)$, by~\eqref{Equ:divide-2} we have
  $$ \begin{cases}
  \lambda(b_{\sigma}) \mid \lambda(x),  &  \text{if $(X,\lambda)$ is of ascending type}; \\
 \lambda(x) \mid \lambda(b_{\sigma}),  &  \text{if $(X,\lambda)$ is of descending type}.
 \end{cases} $$
   So the deformation retraction of $\mathrm{St}(\sigma,K)$ to $b_{\sigma}$ along the line segment from
   any point $x\in \mathrm{St}(\sigma,K)$ to $b_{\sigma}$ is 
   W-continuous. Therefore, the constant map 
  $\mathrm{St}(\sigma,K) \rightarrow b_{\sigma}$
  is a W-homotopy equivalence.    
    In the following, we also consider
   $(b_{\sigma}, \lambda|_{b_{\sigma}})$ as a weighted singular $0$-simplex of $(\mathrm{St}(\sigma,K), \lambda)$, denoted by $\theta_{b_{\sigma}}$.\n
   
  The map $\eta$ is surjective since $\eta(\theta_{b_{\sigma}}) = \sigma$. So it remains to prove 
   $$\mathrm{ker}(\eta) = \mathrm{Im}(\partial^{\lambda}) \ \text{at}\ \mathcal{W}_0\left(\mathrm{St}(\sigma,K),\lambda \right).$$
   Assume $(K,\mu)$ is of ascending type in the rest of the proof (the descending type case is parallel).
    For a weighted singular $1$-simplex 
   $\theta_1: (|\Delta^1|, \lambda_{\widetilde{\xi}}) \rightarrow \left(\mathrm{St}(\sigma,K),\lambda \right)$ where
   $\Delta^1=[v_0, v_1]$, by definition~\eqref{Equ:Bd-Sing}
   $$ \partial^{\lambda}(\theta_1) = \frac{\widetilde{\xi}(\Delta^1)}{\widetilde{\xi}(v_1)}\theta_1|_{v_1} - \frac{\widetilde{\xi}(\Delta^1)}{\widetilde{\xi}(v_0)}\theta_1|_{v_0} \in \mathcal{W}_0\left(\mathrm{St}(\sigma,K),\lambda \right). $$
   Then by the definition of $\eta$ in~\eqref{Equ:Def-Eta},
 $$ \eta\circ \partial^{\lambda}(\theta_1) =
 \frac{\widetilde{\xi}(\Delta^1)}{\widetilde{\xi}(v_1)} \frac{\widetilde{\xi}(v_1)}{\mu(\sigma)} - \frac{\widetilde{\xi}(\Delta^1)}{\widetilde{\xi}(v_0)} \frac{\widetilde{\xi}(v_0)}{\mu(\sigma)}=0.  $$
 This implies $\mathrm{Im}(\partial^{\lambda})\subseteq 
 \mathrm{ker}(\eta)$. 
 Conversely, if a chain $\alpha=\sum_i k_i\theta^{i}_0$ is 
 in $\mathrm{ker}(\eta)$ where  $\theta^i_0: (|\Delta^0|,\lambda_{\xi_i})
 \rightarrow \left(\mathrm{St}(\sigma,K),\lambda  \right)$ are weighted singular $0$-simplices, then
 \begin{equation} \label{Equ:Sum-0}
  \eta(\alpha) =  \sum_i k_i \frac{\xi_i(\Delta^0)}{\mu(\sigma)}\sigma =0 \ \Longrightarrow \  \sum_i k_i \frac{\xi_i(\Delta^0)}{\mu(\sigma)}=0. 
  \end{equation}
 
 Moreover, since $\theta^i_0(\Delta^0)$ lies in $\mathrm{St}(\sigma,K)$, the line segment between $b_{\sigma}$ and $\theta^i_0(\Delta^0)$ determines a unique weighted singular $1$-simplex $\theta^i_1: (|\Delta^1|,\lambda_{\widetilde{\xi}_i})
 \rightarrow \left(\mathrm{St}(\sigma,K),\lambda  \right)$
 where $ \widetilde{\xi}_i$ is an ascending type divisible weight on $\Delta^1=[v_0, v_1]$ defined by
 $$ \widetilde{\xi}_i (v_0) =\lambda(b_{\sigma})=\mu(\sigma), \ \ \widetilde{\xi}_i (v_1) = \xi_i(\Delta^0). $$
Then $\theta^i_1|_{v_0}=\theta_{b_{\sigma}}$, $\theta^i_1|_{v_1} =\theta^i_0$ and $\widetilde{\xi}_i(\Delta^1)= \widetilde{\xi}_i (v_1)= \xi_i(\Delta^0)$.
So we have
 $$ \partial^{\lambda}(\theta^i_1) =  \theta^i_1|_{v_1} - \frac{\xi_i(\Delta^0)}{\mu(\sigma)}\theta^i_1|_{v_0} =\theta^i_0- \frac{\xi_i(\Delta^0)}{\mu(\sigma)} \theta_{b_{\sigma}}. $$
 $$ \Longrightarrow \  \partial^{\lambda}\Big(\sum_i k_i \theta^i_1 \Big) = \sum_i k_i\theta^{i}_0 - \sum_i k_i \frac{\xi_i(\Delta^0)}{\mu(\sigma)} \theta_{b_{\sigma}}\overset{\eqref{Equ:Sum-0}}{=} \alpha. \qquad $$
So $ 
 \mathrm{ker}(\eta) \subseteq \mathrm{Im}(\partial^{\lambda}) $ and the lemma is proved.
    Note that the lemma may fail if $(K,\mu)$ is not divisibly weighted.
\end{proof}

Let $(K,\mu)$ be a divisibly weighted triangulation of a weighted polyhedron $(X,\lambda)$.  For any $n$-simplex $\sigma$ of $K$, the characteristic map
$\theta^{\sigma}: |\Delta^n|\rightarrow |K|\cong X$ of
 $\sigma$
canonically determines a weighted singular $n$-simplex of $(X,\lambda)$, denoted by 
$$\theta^{\sigma}: (|\Delta^n|,\lambda_{\mu,\sigma}) \rightarrow (X,\lambda),$$
where $\lambda_{\mu,\sigma}$ is the pull-back of the weight function $\lambda_{\mu}$ to $\Delta^n$ by $\theta^{\sigma}$ .
 Then we obtain a linear map
 \begin{equation} \label{Equ:Def-Xi} 
 \Xi^K : \big( C_*(K),\partial^{\mu} \big) \rightarrow \big(\mathcal{W}_*(X,\lambda),\partial^{\lambda} \big),
 \ \sigma\mapsto \theta^{\sigma}.
 \end{equation}
 
 It is clear that $\Xi^K$ is a chain map and hence induces
 a homomorphism
 \begin{equation} \label{Equ:Xi-star}
   \Xi^K_*: H_*\big( C_*(K),\partial^{\mu} \big) \rightarrow H_*\big(\mathcal{W}_*(X,\lambda), \partial^{\lambda} \big)=\mathcal{H}^{W}_*(X,\lambda).
   \end{equation}

  \begin{figure}[h]
          % Requires \usepackage{graphicx}
         \includegraphics[width=0.6\textwidth]{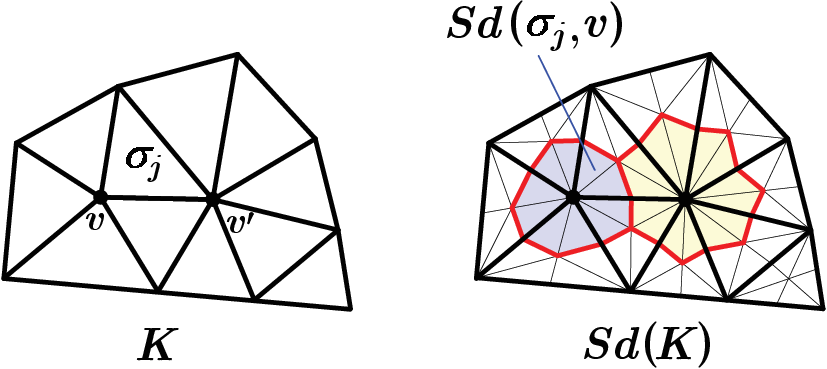}\\
          \caption{Barycentric Subdivision of $K$}\label{p:Barycentric-Subdivision}
      \end{figure}

\begin{thm} \label{Thm:Isom-Char} 
  The isomorphism from
$H_*(K,\partial^{\mu})$ to $\mathcal{H}^{W}_*(X,\lambda)$
 in Theorem~\ref{Thm:Main-Isom}
agrees with $\Xi^K_*$.
\end{thm}
 \begin{proof}
 Let $\alpha = \sum^k_{j=1} \sigma_j$ be an $n$-cycle
 in $ \big( C_*(K),\partial^{\mu} \big)$ where $\sigma_j$
 and $\sigma_{j'}$ are allowed to be the same simplex.   
 So by definition, 
 $\Xi^K(\alpha)= \sum^k_{j=1}  \theta^{\sigma_j}$.
   Consider the barycentric subdivision of $K$ and choose 
  an open cover $\mathcal{U}$ of $X$ defined by~\eqref{Equ:U-Cover}.
  By the proof of Lemma~\ref{Lem:Inclusion-Equiv} in the appendix, we can identify the homology class $[\Xi^K(\alpha)]\in \mathcal{H}^{W}_*(X,\lambda)$ with the homology class of the
  following $\mathcal{U}$-small chain:
   $$  \beta_{-1} := \Xi^{Sd(K)} \Big( \sum^k_{j=1}  Sd_{\#}(\sigma_j) \Big) .$$  
   A preimage of $\beta_{-1}$
   under $\varepsilon$ is distributing its components to the vertices of $K$ that appear in the expression of $\alpha$, that is (see Figure~\ref{p:Barycentric-Subdivision}):
      $$ \beta_0 := \bigg( \Xi^{Sd(K)} \Big(\underset{1\leq j \leq k}{\sum_{v\in \sigma_j}} Sd_{\#}(\sigma_j,v) \Big) \bigg)_{v} \in \underset{\dim(\sigma)=0}{\bigoplus}\, \mathcal{W}_n\left( \mathrm{St}(\sigma,K),\lambda\right),$$
    where the notation $Sd_{\#}(\sigma_j,v)$ is defined in~\eqref{Equ:Sd-sigma-tau}. \n
   
   Since $\alpha$ is a cycle, by our construction $\partial^{\lambda}(\beta_{-1})=0$ and hence
   $\varepsilon \circ \partial^{\lambda}(\beta_0) =0$.
   So in the horizontal direction $\partial^{\lambda}(\beta_0)$ is a boundary chain with respect to $d$. 
 Let
   $$\beta_1:= \bigg( \Xi^{Sd(K)} \Big(\underset{1\leq j \leq k}{\sum_{\overline{vv'}\subseteq \sigma_j}} Sd_{\#}(\sigma_j, \overline{vv'}) \Big) \bigg)_{\overline{vv'}} \in \underset{\dim(\sigma)=1}{\bigoplus}\, \mathcal{W}_{n-1}\left( \mathrm{St}(\sigma,K),\lambda\right),$$
      It is easy to check that 
    $$\partial^{\lambda}(\beta_0) =d(\beta_1), \  \  
    d\circ \partial^{\lambda}(\beta_1) =0.$$
    By iterating the above argument, we obtain a chain
    for each $0\leq s \leq n$ as follows:
      $$\beta_s:= \bigg( \Xi^{Sd(K)} \Big(\underset{1\leq j \leq k}{\sum_{\tau\subseteq \sigma_j}} Sd_{\#}(\sigma_j, \tau) \Big) \bigg)_{\dim(\tau)=s} \in \underset{\mathrm{dim}(\sigma)=s}{\bigoplus}\, \mathcal{W}_{n-s}\left( \mathrm{St}(\sigma,K),\lambda\right), $$
  which satisfies $\partial^{\lambda}(\beta_{s-1}) =d(\beta_s)$ and $  
    d \circ \partial^{\lambda}(\beta_s) =0$ (see Figure~\ref{p:Isom-Homology}).
    In particular, the $0$-chain $\beta_n$ is nothing but the sum of all the barycenters of $\sigma_j$. So by definition, $\eta(\beta_n)=\alpha$. This implies that
    $\Xi_*([\alpha])$ is the isomorphic image of the homology class $[\alpha]$
   under the isomorphism given in Theorem~\ref{Thm:Main-Isom}.   
 \end{proof}
            
  \begin{figure}[h]
          % Requires \usepackage{graphicx}
         \includegraphics[width=0.44\textwidth]{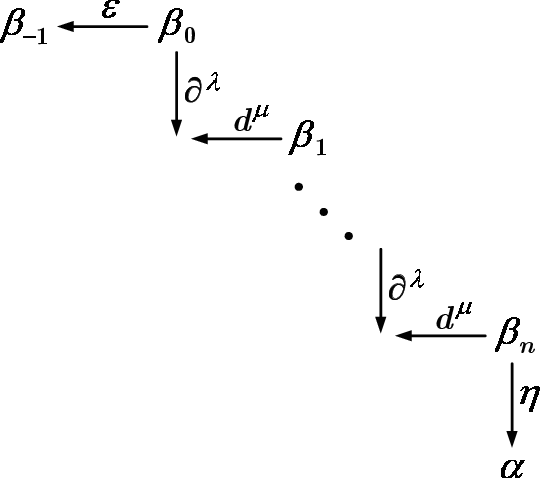}\\
          \caption{Isomorphism between weighted simplicial homology and weighted singular homology}\label{p:Isom-Homology}
      \end{figure}

\subsection{Weight-preserving singular chains}
\ \n
 Let $(\mathcal{WP}_*(X, \lambda), \partial^{\lambda})$ 
 denote the subchain complex of $(\mathcal{W}_*(X, \lambda), \partial^{\lambda})$ generated by all the weighted singular simplices 
$\theta:(|\Delta^n|, \lambda_{\xi}) \to (X, \lambda)$, $n\geq 0$,
which are weight-preserving. We call elements in $\mathcal{WP}_*(X, \lambda)$ the \emph{weight-preserving singular chains} of $(X, \lambda)$. Moreover, denote the homology groups of $(\mathcal{WP}_*(X, \lambda), \partial^{\lambda})$ by $\mathcal{H}_*^{WP}(X, \lambda)$. \n

Let $i: (\mathcal{WP}_*(X, \lambda), \partial^{\lambda}) \to (\mathcal{W}_*(X, \lambda), \partial^{\lambda})$ denote the inclusion map. 
Note that the map $\Xi^K$ defined in \eqref{Equ:Def-Xi}  sends a simplex $\sigma\in K$ to a weighted singular simplex $\theta^{\sigma}$ which is actually weight-preserving. 
Hence, $\Xi^K$ factors as 
$\Xi^K = i\circ \widetilde{\Xi}^K$ where $\widetilde{\Xi}^K:  \big( C_*(K),\partial^{\mu} \big) \rightarrow \big(\mathcal{WP}_*(X,\lambda),\partial^{\lambda} \big)
 $ is defined in the same manner as $\Xi^K$. 

\begin{thm}\label{thm:iso-weight-preserving}
  If $(K,\mu)$ is a divisibly weighted triangulation of $(X,\lambda)$, then the map $\widetilde{\Xi}^K$ induces an isomorphism $\widetilde{\Xi}^K_*$ from  $H_*(K, \mu)$ to $\mathcal{H}_*^{WP}(X, \lambda)$. So the inclusion $i: (\mathcal{WP}_*(X, \lambda), \partial^{\lambda}) \to (\mathcal{W}_*(X, \lambda), \partial^{\lambda})$ induces an isomorphism from $\mathcal{H}_*^{WP}(X, \lambda)$ to  $\mathcal{H}_*^{W}(X, \lambda)$.
\end{thm}

The proof of Theorem~\ref{thm:iso-weight-preserving} is almost identical to the proof of Theorem~\ref{Thm:Isom-Char} by replacing $\mathcal{W}_*(X, \lambda)$ by $\mathcal{WP}_*(X, \lambda)$ throughout (note that Lemma~\ref{Lem:Inclusion-Equiv} also holds for
$\mathcal{WP}_*(X, \lambda)$). The only new argument we need to give is establishing the counterpart of Lemma~\ref{Lem:Retract-Barycenter}. This is because the deformation retraction from $\mathrm{St}(\sigma,K)$ to $b_{\sigma}$ may not be weight-preserving, hence may not
induce a map on $ \mathcal{WP}_*\left(\mathrm{St}(\sigma,K),\lambda_{\mu} \right)$. The following lemma is
what we need for Theorem~\ref{thm:iso-weight-preserving}.

\begin{lem}
    Let $(K, \mu)$ be a divisibly weighted simplicial complex. For any $k$-simplex $\sigma$ of $K$, 
    the following chain complex is exact for $\mathrm{St}(\sigma,K)\subseteq |K|$:
   $$  \cdots  \overset{\partial^{\lambda_{\mu}}}{\longrightarrow}  \mathcal{WP}_1\left(\mathrm{St}(\sigma,K),\lambda_{\mu} \right) \overset{\partial^{\lambda_{\mu}}}{\longrightarrow} \mathcal{WP}_0\left(\mathrm{St}(\sigma,K),\lambda_{\mu} \right)  \overset{\eta}{\longrightarrow}  C_k(K)\longrightarrow 0. $$    
\end{lem}

\begin{proof}
   We only prove the case where $\mu$ is of ascending type since the descending type case is parallel.  The proof of the exactness at $\mathcal{WP}_0\left(\mathrm{St}(\sigma,K),\lambda_{\mu} \right)$ is similar to Lemma~\ref{Lem:Retract-Barycenter}. So we only need to prove 
       $ \mathcal{H}^{WP}_n(\mathrm{St}(\sigma,K), \lambda_{\mu}) = 0$ for all $n\geq 1$.\n

   Let $p$ be a vertex of $\sigma$ whose weight is minimal among all the vertices of $\sigma$. Then $p$ has the minimal weight among all points of $\mathrm{St}(\sigma,K)$ since $\mu$ is ascending. Suppose a weighted singular simplex $\theta: (|\Delta^n|, \lambda_{\xi}) \to (\mathrm{St}(\sigma,K), \lambda_{\mu})$, $n\geq 1$, is weight-preserving. Then there is a unique map $\theta^p: |\Delta^{n+1}| \to \mathrm{St}(\sigma,K)$ with $\theta^p(v_0)=p$ and the restriction
    of $\theta^p$ to $[v_1,\cdots, v_{n+1}]$ agrees with $\theta$; see~\eqref{Equ:Cone-theta} in the appendix for the details. \n
     
        Moreover, let $\xi'$ be a weight function on $\Delta^{n+1}=[v_0, \cdots, v_{n+1}]$ defined by:
    $$ \xi'(v_{0}):= \lambda_{\mu}(p), \ \ \xi'(v_i):=\xi(v_i), \ 1\leq i \leq n+1.$$
     
     Clearly, $\xi'$ is a divisible weight on $\Delta^{n+1}$ and $\theta^p: (|\Delta^{n+1}|, \lambda_{\xi'})\rightarrow (\mathrm{St}(\sigma,K), \lambda_{\mu})$ is weight-preserving. Similarly to~\eqref{Equ:Cone-p}, here we denote $\theta^p$ by $p\cdot \theta$.    
     Moreover, for a chain $\alpha=\sum k_i\theta_i \in 
      \mathcal{WP}_n\left(\mathrm{St}(\sigma,K),\lambda_{\mu} \right)$, define $p\cdot \alpha =\sum k_i (p\cdot \theta_i)$.\n
     
     Since $\xi'(v_{0})$ is minimal among $\xi'(v_i)$, $0\leq i\leq n+1$, we have 
    \begin{equation*}
        \xi'(\Delta^{n+1}) = \xi' (\partial_{0}\Delta^{n+1}) = \xi(\Delta^n); \quad \xi'(\partial_j \Delta^{n+1}) = \xi(\partial_{j+1} \Delta^{n}), \  0 \leq j \leq n.
    \end{equation*}
    Then by definition,
    \begin{equation*}
    \partial^{\lambda_{\mu}} (p\cdot\theta) =    \partial^{\lambda_{\mu}} \theta^p= \sum_{j=0}^{n+1} (-1)^j \frac{\xi'(\Delta^{n+1})}{\xi'( \partial_j \Delta^{n+1})} \cdot \theta^p|_{\partial \Delta^{n+1}_j} =   \theta - p\cdot (\partial^{\lambda_{\mu}} \theta).
    \end{equation*}
    Therefore, for any $n$-chain $\alpha \in \mathcal{WP}_n((\mathrm{St}(\sigma,K), \lambda_{\mu}))$, we have
    $$  \partial^{\lambda_{\mu}} (p\cdot\alpha) =  
    \alpha - p\cdot (\partial^{\lambda_{\mu}} \alpha).  $$
   In particular, if $\alpha$ is a closed chain, i.e. $\partial^{\lambda_{\mu}} \alpha = 0$, then $\partial^{\lambda_{\mu}} (p\cdot\alpha) = \alpha$. This
 means $ \mathcal{H}^{WP}_n(\mathrm{St}(\sigma,K), \lambda_{\mu}) = 0$. The lemma is proved.
\end{proof}

\vskip 0.4cm

  \section{Product structure on weighted singular cohomology} \label{Sec:Product-Cohomology}
  
   Let $(X,\lambda)$ be a weighted polyhedron. Given
   an abelian group $G$, the \emph{weighted singular cochain complex} of $(X,\lambda)$ with $G$-coefficients is obtained by applying
  the $\mathrm{Hom}(-, G)$ functor to 
  the $(\mathcal{W}_*(X,\lambda),\partial^{\lambda})$, denoted by $(\mathcal{W}^*(X,\lambda;G),\delta^{\lambda})$,
  where $\delta^{\lambda}$ is the coboundary map determined by $\partial^{\lambda}$.
  Then the \emph{weighted singular cohomology group of $(X,\lambda)$} with $G$-coefficients is defined to be
   $$ \mathcal{H}^*_W(X,\lambda;G):= H^*(\mathcal{W}^*(X,\lambda;G),\delta^{\lambda}).$$ 
    
    If $f: (X,\lambda) \rightarrow (X',\lambda')$ is a
   W-continuous map, then we obtain a cochain map
  $$f^{\#}: (\mathcal{W}^*(X',\lambda';G),\delta^{\lambda'})\rightarrow (\mathcal{W}^*(X,\lambda;G),\delta^{\lambda}) $$ 
 by: for any $\phi\in \mathcal{W}^*(X',\lambda';G)$ and any weighted singular simplex $\theta\in \mathcal{W}_n(X,\lambda)$,
  $$f^{\#}(\phi)(\theta):= 
  \phi(f_{\#}(\theta).$$
 Moreover, $f^{\#}$ induces a homomorphism on the weighted singular cohomology
 $$ f^*: \mathcal{H}^*_W(X',\lambda';G) \rightarrow  
 \mathcal{H}^*_W(X,\lambda;G). $$

   \subsection{Weighted cup product}
  \ \n
   We find that there is a natural product structure on
  the weighted singular cohomology of a descending type weighted polyhedron, which generalizes the cup product $\cup$ of ordinary singular cohomology defined by Alexander-Whitney diagonal. Moreover, we will prove  that this product is graded commutative.  
  The coefficients of the weighted cohomology groups in this section are 
 always assumed to be $\Z$ if not specified otherwise.\n
  
 \begin{defi}[Weighted Cup Product] \label{Def:Product-DW}
  \ \n
  Let $(X,\lambda)$ be a descending type weighted polyhedron and $R$ be a commutative ring with unit.
  Then for any cochains $\phi\in \mathcal{W}^p(X,\lambda)$ and
  $\psi\in \mathcal{W}^q(X,\lambda)$, let $\phi \Cup \psi\in
  \mathcal{W}^{p+q}(X,\lambda)$
  be the cochain whose value on a
   weighted singular $(p+q)$-simplex
  $\theta: (|\Delta^{p+q}|,\lambda_{\xi}) \rightarrow (X,\lambda)$ is defined by
  \begin{align}  \label{Equ:Prod-Wt-Homol}
   \phi\Cup \psi  (\theta)
   :=  \frac{\xi([v_0,\cdots,v_p]) \xi([v_p,\cdots,v_{p+q}])}{\xi([v_0,\cdots,v_{p+q}])} \, \phi(\theta|_{[v_0,\cdots,v_p]})
 \cdot \psi(\theta|_{[v_p,\cdots,v_{p+q}]}).
 \end{align}
 \normalsize
 The coefficient on the right hand side of~\eqref{Equ:Prod-Wt-Homol} is always integral because $\xi$ is descending.
  We call $\Cup$ the \emph{weighted cup product} on $(\mathcal{W}^*(X,\lambda),\delta^{\lambda})$. 
 \end{defi}

 \begin{lem} \label{Lem:Cobound}
 Let $(X,\lambda)$ be a descending type weighted polyhedron. For any cochains $\phi\in \mathcal{W}^p(X,\lambda)$ and $\psi\in \mathcal{W}^q(X,\lambda)$, we have
   \begin{equation} \label{Equ:Cobound-formula}
   	 \delta^{\lambda}(\phi\Cup \psi) =\delta^{\lambda} \phi \Cup\psi + (-1)^p\phi \Cup \delta^{\lambda}\psi.
   \end{equation}  
  \end{lem}
  \begin{proof}
   For a weighted singular $(p+q+1)$-simplex
  $\theta: (|\Delta^{p+q+1}|,\lambda_{\xi}) \rightarrow (X,\lambda)$,
  \small
\begin{align*}
  &\ \delta^{\lambda}(\phi\Cup \psi) (\theta)  = \phi\Cup \psi \big(\partial^{\lambda} (\theta ) \big) \\
  = &\, \phi\Cup \psi \Big( \sum^{p+q+1}_{i=0} (-1)^i \frac{\xi([v_0,\cdots,\widehat{v}_i,\cdots,v_{p+q+1}])}{\xi([v_0,\cdots,v_{p+q+1}])}\, \theta|_{[v_0,\cdots,\widehat{v}_i,\cdots,v_{p+q+1}]}\Big) \\
  = &   \sum^{p}_{i=0} (-1)^i         
       \frac{\xi([v_0,\cdots,\widehat{v}_i,\cdots,v_{p+1}])\xi([v_{p+1},\cdots,v_{p+q+1}])}{\xi([v_{0},\cdots,v_{p+q+1}])}\,  \phi\big(\theta|_{[v_0,\cdots,\widehat{v}_i,\cdots,v_{p+1}]} \big) 
  \cdot \psi\big(\theta|_{[v_{p+1},\cdots,v_{p+q+1}]} \big)\\
  +& \sum^{p+q+1}_{i=p+1} (-1)^i \frac{\xi([v_0,\cdots,v_p])        
        \xi([v_p,\cdots,\widehat{v}_i,\cdots, v_{p+q+1}])}{\xi([v_{0},\cdots,v_{p+q+1}])}\,  \phi\big(\theta|_{[v_0,\cdots,v_p]} \big) 
  \cdot \psi \big(\theta|_{[v_p,\cdots, \widehat{v}_i,\cdots,v_{p+q+1}]} \big).
  \end{align*}
  
  \normalsize
  On the other side, we have
  \small
   \begin{align*}
       &\quad  \ (\delta^{\lambda}\phi\Cup \psi)  (\theta) \\
       &= \frac{\xi([v_0,\cdots,v_{p+1}])        
        \xi([v_{p+1},\cdots,v_{p+q+1}])}{\xi([v_{0},\cdots,v_{p+q+1}])}\, \delta^{\lambda}\phi \big(\theta|_{[v_0,\cdots,v_{p+1}]} \big) 
  \cdot \psi \big(\theta|_{[v_{p+1},\cdots,v_{p+q+1}]}\big)
     \end{align*}  
 \normalsize
   where 
 \small 
 \begin{align*}
    \delta^{\lambda}\phi \big(\theta|_{[v_0,\cdots,v_{p+1}]} \big) 
      & =\sum^{p+1}_{i=0} (-1)^i \frac{\xi([v_0,\cdots,\widehat{v}_i,\cdots,v_{p+1}])}{\xi([v_0,\cdots,v_{p+1}])}
      \, \phi \big(\theta|_{[v_0,\cdots,\widehat{v}_i,\cdots,v_{p+1}]} \big).
  \end{align*}
 \normalsize
  Then we have
 \small
  \begin{align*}
       &\quad  \ (\delta^{\lambda} \phi\Cup \psi) (\theta)   \\
       &= \sum^{p+1}_{i=0} (-1)^i         
       \frac{\xi([v_0,\cdots,\widehat{v}_i,\cdots,v_{p+1}])\xi([v_{p+1},\cdots,v_{p+q+1}])}{\xi([v_{0},\cdots,v_{p+q+1}])}\,  \phi \big(\theta|_{[v_0,\cdots,\widehat{v}_i,\cdots,v_{p+1}]} \big) 
  \cdot \psi \big(\theta|_{[v_{p+1},\cdots,v_{p+q+1}]} \big).
     \end{align*} 
      \normalsize
     Similarly, we have 
  \small
   \begin{align*}
       &\quad  \ (-1)^p (\phi\Cup \delta^{\lambda}\psi)(\theta) \\
       &= \sum^{p+q+1}_{i=p} (-1)^i \frac{\xi([v_0,\cdots,v_p])        
        \xi([v_p,\cdots,\widehat{v}_i,\cdots, v_{p+q+1}])}{\xi([v_{0},\cdots,v_{p+q+1}])}\,  \phi \big(\theta|_{[v_0,\cdots,v_p]} \big) 
  \cdot \psi \big(\theta|_{[v_p,\cdots, \widehat{v}_i,\cdots,v_{p+q+1}]} \big).
     \end{align*} 
      \normalsize
  
  If we add the above two expressions up, the last term of the first sum cancels the first term of the second sum, and the remaining terms give exactly $\delta^{\lambda}(\phi\Cup \psi)(\theta) $.
  The lemma is proved.
  \end{proof}
  
  By Lemma~\ref{Lem:Cobound}, the above product $\Cup$ on cochains induces the weighted cup product on 
  $\mathcal{H}^*_W(X,\lambda)$:
   $$ \mathcal{H}^p_W(X,\lambda) \times
   \mathcal{H}^q_W(X,\lambda) \overset{\Cup}{\longrightarrow} \mathcal{H}^{p+q}_W(X,\lambda), \ p,q\in \Z.  $$
 
  Moreover, it is easy to check that the weighted cup product $\Cup$ is \emph{natural}
 in the sense that for 
 any  W-continuous map $f: (X,\lambda) \rightarrow (X',\lambda')$  between two descending type weighted polyhedra $(X,\lambda)$ and $(X',\lambda')$, the following diagram commutes
 \[ \xymatrix{
         \mathcal{H}^p_W(X',\lambda') \times
   \mathcal{H}^q_W(X',\lambda') \ar[d]_{f^*\times f^*} \ar[r]^{\qquad\quad \Cup}
                &  \mathcal{H}^{p+q}_W(X',\lambda') \ar[d]^{f^*}  \\
          \mathcal{H}^p_W(X,\lambda) \times
   \mathcal{H}^q_W(X,\lambda) \ar[r]^{\qquad\quad \Cup} &  \mathcal{H}^{p+q}_W(X,\lambda)
                 }.  \]

 \begin{rem}
 	 It seems to us that there is no meaningful product structure on the weighted singular cohomology for an ascending type weighted polyhedron. Indeed, the naive extension of the formula in~\eqref{Equ:Prod-Wt-Homol} to an ascending type weighted polyhedron cannot guarantee the coefficients
 	 in the formula to be integral and make the coboundary formula in~\eqref{Equ:Cobound-formula} hold simultaneously.
 \end{rem}  
 
 \subsection{Graded commutativity of the weighted cup product} \label{Subsec:Graded-Comm}
\ \n
  Now we are ready to prove that the weighted cup product
   $\Cup$
 on the weighted singular cohomology is graded commutative. The proof is parallel to the proof of graded commutativity of
 singular cohomology in~\cite[Theorem 3.11]{Hatcher02}.
 \n
 
 First of all, for a weighted polyhedron $(X,\lambda)$, we introduce an auxiliary map $\rho$ on 
 $\mathcal{W}_*(X,\lambda)$. For the standard $n$-simplex $\Delta^n=[v_0,\cdots, v_n]$, define a simplicial homeomorphism 
 $$T: \Delta^n\rightarrow \Delta^n, \ v_i\mapsto v_{n-i},\
  0 \leq i \leq n.$$

    For a divisible weight $\xi$ on $\Delta^n$, let $\overline{\xi}$ be the pull-back of $\xi$ to $\Delta^n$
    by $T$, i.e. 
    $$\overline{\xi}(v_i) = \xi(T (v_i))=\xi(v_{n-i}), \ 0\leq i \leq n.$$
    So $T: (\Delta^n,\overline{\xi}) \rightarrow (\Delta^n,\xi)$ is an isomorphism of weighted simplicial complexes, whose geometric realization is (see Example~\ref{Exam:Geo-Realiztion}) 
    $$ |T|: (|\Delta^n|, \lambda_{\overline{\xi}}) \rightarrow (|\Delta^n|, \lambda_{\xi}).$$
  
  For a weighted singular $n$-simplex
  $\theta: (|\Delta^n|,\lambda_{\xi}) \rightarrow (X,\lambda)$, we have an associated  weighted singular $n$-simplex
  \[ \ \overline{\theta}=\theta\circ |T|: (|\Delta^n|,\lambda_{\overline{\xi}}) \rightarrow (|\Delta^n|,\lambda_{\xi}) \rightarrow (X,\lambda) \]
  Moreover, define a linear map $\rho:  \big(\mathcal{W}_*(X,\lambda), \partial^{\lambda} \big) \rightarrow \big(\mathcal{W}_*(X,\lambda), \partial^{\lambda} \big)$  by
 \begin{equation} \label{Equ:rho-Def}
  \rho ( \theta) : = \varepsilon_n \overline{\theta}, \ \text{where} \ \varepsilon_n=(-1)^{\frac{n(n+1)}{2}}. 
 \end{equation}
 
 \n
 \begin{lem} \label{Lem:Inver-Order}
 The map $\rho:  \big(\mathcal{W}_*(X,\lambda), \partial^{\lambda} \big) \rightarrow \big(\mathcal{W}_*(X,\lambda), \partial^{\lambda} \big)$ is a chain map 
  and $\rho$ is chain homotopic to the identity map.  
 \end{lem}
 \begin{proof}
First, we prove that $\rho$ is a chain map, i.e. $\partial^{\lambda} \rho = \rho \partial^{\lambda}$. we compute
\begin{equation*}
    \begin{aligned}
        \partial^{\lambda} \rho (\theta) &= \partial^{\lambda}(\varepsilon_n \overline{\theta}) \\ 
    &=\sum_{i=0}^n (-1)^i \varepsilon_n \frac{\overline{\xi}([v_0, \cdots, \hat{v}_i, \cdots, v_n])}{\overline{\xi}([v_0, \cdots, v_n])} \overline{\theta}|_{[v_0, \cdots, \hat{v}_i, \cdots, v_n]} \\ 
    &=\sum_{i=0}^n (-1)^i \varepsilon_n \frac{\xi([v_n, \cdots, \hat{v}_{n-i}, \cdots, v_n])}{\xi([v_n, \cdots, v_0])} \theta|_{[v_n, \cdots, \hat{v}_{n-i}, \cdots, v_n]},
    \end{aligned}
\end{equation*}
and
\begin{equation*}
    \begin{aligned}
        \rho\partial^{\lambda}(\theta) &= \rho \Bigl( \sum_{j=0}^n (-1)^j\frac{\xi([v_0,\cdots, \hat{v}_j, \cdots,v_n])}{\xi([v_0,\cdots, v_n])} \theta|_{[v_0,\cdots, \hat{v}_j, \cdots,v_n]}\Bigr) \\ 
        &= \sum_{j=0}^n (-1)^j\varepsilon_{n-1} \frac{\xi([v_0,\cdots, \hat{v}_j, \cdots,v_n])}{\xi([v_0,\cdots, v_n])} \theta|_{[v_n, \cdots, \hat{v}_j, \cdots, v_0]}. \quad
    \end{aligned}
\end{equation*}
Since 
$$
    (-1)^{n-i}\varepsilon_{n-1} = (-1)^{n-i + \frac{n(n-1)}{2}} = (-1)^{-i} \varepsilon_n = (-1)^i \varepsilon_n,$$
the $(n-i)$-th term in the summation in the second equation equals the $i$-th term in the summation in the first equation. Therefore, $\partial^{\lambda} \rho = \rho \partial^{\lambda}$.\n
 
Next, we prove that $\rho$ is chain homotopic to the identity.
 For a divisibly weighted simplex $(\Delta^n, \xi)$ of descending type, consider $(\Delta^n \times [0, 1], \xi \times \mathbf{1})$. We identify $\Delta^n$ with $\Delta^n \times \{0\}$ and use $v_i$ to denote the $i$-th vertex of either simplex, and use $w_i$ to denote the $i$-th vertex of $\Delta^n \times \{1\}$.
Let $\pi: \Delta^n \times [0, 1] \to \Delta^n$ denote the projection map. Define $P: \mathcal{W}_n(X, \lambda) \rightarrow \mathcal{W}_{n+1}(X, \lambda)$, $n\in \Z$, by
\begin{equation}
    P(\theta):= \sum_{i=0}^n (-1)^i \varepsilon_{n-i}(\theta \circ \pi)|_{[v_0, \cdots, v_i, w_n, \cdots, w_i]}. 
\end{equation}
 By direct computation, we have 
 \begin{equation*}
     \begin{aligned}
         \partial^{\lambda} P(\theta) &= 
         \sum_{j\leq i} (-1)^{i+j} \varepsilon_{n-i}\frac{\xi([v_0, \cdots, \hat{v}_j, \cdots, v_n])}{\xi([v_0, \cdots, v_n])} (\theta \circ \pi)|_{[v_0, \cdots, \hat{v}_j, \cdots, v_i, w_n, \cdots, w_i]} \\ 
         & +  \sum_{i\leq j } (-1)^{n+1-j} \varepsilon_{n-i}\frac{\xi([v_0, \cdots, \hat{v}_j, \cdots, v_n])}{\xi([v_0, \cdots, v_n])} (\theta \circ \pi)|_{[v_0, \cdots, v_i, w_n, \cdots, \hat{w}_j, \cdots, w_i]} \\ 
     \end{aligned}
 \end{equation*}
 \begin{equation*}
     \begin{aligned}
         P \partial^{\lambda} (\theta) &= P\Big(\sum_{j=0}^n (-1)^j \frac{\xi([v_0, \cdots, \hat{v}_j, \cdots, v_n])}{\xi([v_0, \cdots, v_n])} \theta|_{[v_0, \cdots, \hat{v}_j, \cdots, v_n]} \Big)\\
         & = \sum_{i<j} (-1)^{i+j} \varepsilon_{n-1-i} \frac{\xi([v_0, \cdots, \hat{v}_j, \cdots, v_n])}{\xi([v_0, \cdots, v_n])}  (\theta\circ \pi)|_{[v_0, \cdots, v_i, w_n, \cdots, \hat{w}_j, \cdots, w_i]} \\ 
         & + \sum_{j<i} (-1)^{i+j+1} \varepsilon_{n-i} \frac{\xi([v_0, \cdots, \hat{v}_j, \cdots, v_n])}{\xi([v_0, \cdots, v_n])}  (\theta\circ \pi)|_{[v_0, \cdots, \hat{v}_j, \cdots, v_i, w_n, \cdots, w_i]}.
     \end{aligned}
 \end{equation*} 
 Adding the above two equations, we obtain 
 \begin{equation*}
     \begin{aligned}
         &\partial^{\lambda} P(\theta) + P \partial^{\lambda} (\theta)\\
 =& \sum_{i=0}^n  \varepsilon_{n-i}(\theta \circ \pi)|_{[v_0, \cdots, v_{i-1}, w_n, \cdots, w_i]} 
        +  \sum_{j=0}^n (-1)^{n+1-j} \varepsilon_{n-j} (\theta \circ \pi)|_{[v_0, \cdots, v_j, w_n, \cdots, \cdots, w_{j+1}]} \\
         =& \sum_{i=0}^n  \varepsilon_{n-i}(\theta \circ \pi)|_{[v_0, \cdots, v_{i-1}, w_n, \cdots, w_i]} 
         +  \sum_{i=1}^{n+1} (-1)^{n-i} \varepsilon_{n-i+1} (\theta \circ \pi)|_{[v_0, \cdots, v_{i-1}, w_n, \cdots, \cdots, w_{i}]} \\
     \end{aligned}
 \end{equation*}
 Since $(-1)^{n-i}\varepsilon_{n-i+1}=-\varepsilon_{n-i}$, the above equation equals
 \begin{equation*}
     \varepsilon_n (\theta \circ \pi)|_{[w_n, \cdots, w_0]} - (\theta \circ \pi)|_{[v_0, \cdots, v_{n}]} = \rho(\theta) - \theta.
 \end{equation*}
 Therefore, $P$ is a chain homotopy between $\rho$ and the identity map. The lemma is proved.
 \end{proof}

 By the above lemma, $\rho$ induces a cochain map 
 \begin{equation*}
     \rho^{\sharp}: (\mathcal{W}^*(X, \lambda), \delta^{\lambda}) \to (\mathcal{W}^*(X, \lambda), \delta^{\lambda}),
 \end{equation*}
 which is cochain homotopic to the identity map. Therefore, the map $\rho^*$ induced by $\rho$ on the cohomology group is the identity map. 
The following lemma is straightforward from Lemma~\ref{Lem:Inver-Order}.

 \begin{lem}\label{lem-graded-com}
     For any cochains $\phi \in \mathcal{W}^p(X, \lambda)$ and $\psi \in \mathcal{W}^q(X, \lambda)$, we have 
     \begin{equation*}
         \rho^{\sharp} (\psi \Cup \phi) = (-1)^{pq} \phi \Cup \psi. 
     \end{equation*}
 \end{lem}

We summarize this section in the following theorem:
 \begin{thm}
     Let $(X, \lambda)$ be a descending type weighted polyhedron. 
     For any cohomology classes $[\phi]\in \mathcal{H}^p_W(X, \lambda)$ and $[\psi]\in \mathcal{H}^q_W(X, \lambda)$, we have 
     \begin{equation}
         [\phi] \Cup [\psi] = (-1)^{pq} [\psi] \Cup [\phi]. 
     \end{equation}
 \end{thm}

 \begin{proof}
     By, \Cref{Lem:Inver-Order} and \Cref{lem-graded-com}, we have 
     \begin{equation*}
         \begin{aligned}
             [\phi] \Cup [\psi] = \rho^*([\phi] \Cup [\psi])= [\rho^{\sharp}(\phi\Cup \psi)] =
             & [(-1)^{pq} \psi \Cup \phi] \\ 
             = & (-1)^{pq} [\psi] \Cup [\phi].  
         \end{aligned}
     \end{equation*}
     So the product $\Cup$ on $\mathcal{H}^p_W(X, \lambda)$ is graded commutative. 
 \end{proof}
   \nn
 
  \subsection{Weighted cap product}
  \ \n
  For a descending type weighted polyhedron $(X, \lambda)$, we can also define a product 
 between elements of $H_*(X, \lambda;R)$
 and $H^*(X, \lambda;R)$ where $R$ is a commutative ring.

 \begin{defi}[Weighted Cap Product]      
 	 \label{Def:Weighted-Cap-Prod}
 Suppose $(X, \lambda)$ is a descending type weighted polyhedron. Define $R$-bilinear \emph{weighted cap product} 
  $$ \Cap:  \mathcal{W}_n(X, \lambda; R)\times \mathcal{W}^p(X, \lambda; R) \rightarrow  \mathcal{W}_{n-p}(X, \lambda; R), \ 0\leq p\leq n $$ 
  by: for any cochain $\phi\in \mathcal{W}^p(X, \lambda; R)$ and any chain $\theta\in \mathcal{W}_n(X, \lambda; R)$, 
   \begin{equation*} 
   \theta  \,\Cap\, \phi := \frac{\xi\big([ v_0,\cdots, v_{p} ]\big)\xi\big([ v_p,\cdots, v_{n} ]\big)}{\xi\big([ v_0,\cdots, v_{n} ]\big)} \phi\big( \theta|_{  [ v_0,\cdots,v_{p} ]} \big) \theta|_{[ v_{p},\cdots,v_{n} ]}. 
   \end{equation*}
 The coefficient on the right hand side is integral 
  because $\xi$ is a descending weight. 
 \end{defi}
 
 For any chain
 $\theta \in \mathcal{W}_n(X, \lambda; R)$,
 it is routine to check from the definitions that 
  \begin{equation} \label{Equ:Cap-Bound-Rel}
    \partial^{\lambda} \big( \theta\,\Cap\,
    \phi \big) = (-1)^p \big(  \partial^{\lambda} \theta\,\Cap\, \phi - 
    \theta\,\Cap\, \delta^{\lambda}\phi \, \big).
  \end{equation}
 This implies that there is an induced product
 \begin{equation*} 
  \mathcal{H}_n^W(X, \lambda;R) \times \mathcal{H}_W^p(X, \lambda;R) \overset{\Cap}{\longrightarrow} 
  \mathcal{H}^W_{n-p}(X, \lambda;R).
 \end{equation*} 
  
 Moreover, the following lemma tells us that 
 the two products $\Cup$ and 
 $\Cap$ are compatible just as the ordinary
 cup product and cap product do.
 
 \begin{lem} \label{Lem:Cap-Cup-Compatible} 
 Let $(X,\lambda)$ be a descending type weighted polyhedron. Then for any cochains $\phi\in \mathcal{W}^p(X, \lambda; R)$,
 $\psi\in \mathcal{W}^q(X, \lambda; R)$ and any chain
 $\theta\in \mathcal{W}_n(X, \lambda; R)$ with $p+q\leq n$, 
  $$ (\theta \,\Cap\, \phi ) \,\Cap\, \psi =  \theta \,\Cap\,
  \big( \phi \,\Cup\, \psi \big). $$
 \end{lem}
 \begin{proof}
  By computation, we obtain 
 \begin{align*}
 &\ \big( \theta  \,\Cap\, \phi \big)  \,\Cap\, \psi\\
 =& \frac{\xi\big([ v_0,\cdots, v_{p} ]\big)\xi\big([ v_p,\cdots, v_{p+q} ]\big) \xi\big([ v_{p+q},\cdots, v_{n} ]\big)}{\xi\big([ v_0,\cdots, v_{n} ]\big)}\phi\big(  \theta|_{[ v_0,\cdots,v_{p}]} \big) 
 \psi\big(  \theta|_{[ v_p,\cdots,v_{p+q} ]} \big) \theta|_{[ v_{p+q},\cdots,v_{n} ]};
 \end{align*} 
 and
 \begin{align*}
 &\  \theta  \,\Cap\, \big( \phi \,\Cup\, \psi \big) \\
 =& \frac{\xi\big([ v_0,\cdots, v_{p+q} ]\big) \xi\big([ v_{p+q},\cdots, v_{n} ]\big)}{\xi\big([ v_0,\cdots, v_{n} ]\big)} 
 \big( \phi \,\Cup\, \psi \big)\big(  \theta|_{[ v_0,\cdots,v_{p+q} ]} \big) 
  \theta|_{[ v_{p+q},\cdots,v_{n} ]}.
 \end{align*}
 Then it is easy see that $\theta \,\Cap\, \big( \phi \,\Cup\, \psi \big) = \big( \theta  \,\Cap\, \phi \big)  \,\Cap\, \psi$.
 \end{proof}
 
By the above lemma, we obtain the following corollary immediately.

\begin{cor} \label{Cor:Cap-Cup-Compatible}
Let $(X, \lambda)$ be a descending type weighted polyhedron. Then for any cohomology classes $[\phi]\in \mathcal{H}_W^p(X, \lambda;R)$, $[\psi]\in \mathcal{H}_W^q(X, \lambda;R)$ and a homology class
 $[\theta]\in \mathcal{H}^W_n(X, \lambda;R)$ with $p+q\leq n$, 
  $$ ([\theta] \Cap [\phi]) \Cap [\psi] =  [\theta] \Cap
  ( [\phi] \Cup [\psi] ). $$
\end{cor}

By the above corollary, the product $\Cap$ induces a right $\mathcal{H}_W^*(X,\lambda;R)$-module structure on $\mathcal{H}^W_*(X,\lambda;R)$.
 \n

\begin{rem} \label{Rem:Cohomology}
Suppose $(K,\mu)$ is a divisibly weighted triangulation of a weighted polyhedron $(X,\lambda)$. Since the chain map
$\Xi^K : \big( C_*(K),\partial^{\mu} \big) \rightarrow \big(\mathcal{W}_*(X,\lambda),\partial^{\lambda} \big)$ induces an isomorphism from $H_*(K,\partial^{\mu})$ to $\mathcal{H}^{W}_*(X,\lambda)$ , we can parallelly define
the simplicial version of weighted cup product and
weighted cap product for $K$. 
\end{rem}
  
  \vskip 0.4cm

  \section{Relation to orbifold theories}
  \label{Sec:Relation}
  In this section, we will show that an orbifold is naturally a weighted polyhedron. Moreover, we will interpret some known theories of orbifolds by our weighted 
  singular homology and cohomology.\n
  
 An orbifold is a topological space which is locally a finite group quotient of a Euclidean space. More specifically,  let $M$ be a paracompact Hausdorff space.      
   An \emph{orbifold chart} on $M$ is given by a connected open subset $\tilde{U}$ of $\R^n$ for some integer $n\geq 0$, a finite group $G$ acting smoothly and effectively on $\tilde{U}$, and a map $\varphi: \tilde{U}\rightarrow M$, such that $\varphi$ is $G$-invariant ($\varphi\circ g=\varphi$ for all $g\in G$) and induces a homeomorphism from $\tilde{U}\slash G$ onto an
 open subset $U=\varphi(\tilde{U})$ of $M$.
 An \emph{orbifold atlas} on $M$ is a family 
 $\mathcal{U}=\{ (\tilde{U},G,\varphi) \}$ of such charts, which cover $M$ and satisfy some local compatible conditions. Two orbifold atlases on $M$ are said to be 
 \emph{equivalent} if they have a common refinement. 
 An \emph{orbifold} (of dimension $n$) is such a space $M$ with an equivalence class of atlases $\mathcal{U}$. 
 By the fact that a smooth action is locally smooth (see Bredon~\cite[p.\,308]{Bredon72}), any
orbifold of dimension $n$ has an atlas consisting of ``linear'' charts, i.e. charts of
the form $(\R^n, G,\varphi)$ where the finite group $G$ acts on $\R^n$ via orthogonal linear transformations. 
 The reader is referred to~\cite{Sa56,Sa57,AdemLeiRuan07,Choi12} for more detailed discussion of these definitions and some other notions 
  in the study of orbifolds.\n 
 
  Let $\mathcal{M} = (M,\mathcal{U})$ be an orbifold of dimension $n$ where $\mathcal{U}$ is an atlas consisting of linear charts. For each point $x \in M$, choose a linear chart $(\R^n, G,\varphi)$ around $x$, with $G$ a finite subgroup of the orthogonal
group $\mathrm{O}(n,\R)$. Let $\tilde{x}$ be a point
with $\varphi(\tilde{x})=x$, and $G_x=\{ g\in G\,|\, g\cdot \tilde{x} = \tilde{x}\}$ the \emph{isotropy subgroup} 
at $\tilde{x}$. Up to conjugation, $G_x$ is a well defined
subgroup of $\mathrm{O}(n,\R)$, called the \emph{local group} at $x$. The order $|G_x|$ of $G_x$ is independent on the local chart around $x$.
   Then $\mathcal{M}$ canonically defines a weighted space $(M,\lambda_{\mathcal{M}})$ where
   for any point $x$ in $M$,
   $$\lambda_{\mathcal{M}}(x) =\text{the order $|G_x|$ of the local group $G_x$ of}\ x.$$

 Next, we show that $(M,\lambda_{\mathcal{M}})$ is a weighted polyhedron. It is well known that (see Goresky~\cite{Gor78} and Verona~\cite{Veron84}) there exists a triangulation $\mathcal{T}$ of $M$ such that the closure of every stratum of $M$ is a simplicial subcomplex of $\mathcal{T}$. Then the relative interior of each simplex of $\mathcal{T}$ is contained in a single stratum of $M$. A detailed proof of this result can be found in Choi~\cite[Section\,4.5]{Choi12}.
By replacing $\mathcal{T}$ by a stellar subdivision, one can assume that
the cover of closed simplices in $\mathcal{T}$ refines the cover of $M$ induced by the atlas $\mathcal{U}$.
For a simplex $\sigma$ in such a triangulation, the isotropy groups of all the interior points
of $\sigma$ are the same, and are subgroups of the isotropy groups of the boundary
points of $\sigma$. By taking a further subdivision of $\mathcal{T}$, we may assume that for any simplex $\sigma$ in $\mathcal{T}$,

\begin{itemize}
 \item[($\scalebox{1.3}{$\star$}$)] there is
one face $\tau$ of $\sigma$ such that the local group is constant on $\sigma\backslash \tau$, and possibly larger
on $\tau$. 
\end{itemize}

We call such a triangulation $\mathcal{T}$ \emph{adapted to $\mathcal{U}$} (see~\cite{MoePro99}). By abuse of notation, we also use $\mathcal{T}$
to refer to the simplicial complex defined by $\mathcal{T}$.  
  \n

  \begin{prop}[{\cite[Proposition 1.2.1]{MoePro99}}]
  For any orbifold $\mathcal{M} = (M,\mathcal{U})$, there
  always exists an adapted triangulation.
\end{prop}

Note that any simplex $\sigma$ in a triangulation $\mathcal{T}$ adapted to $\mathcal{M}$ will have a vertex $v \in \sigma$ with the minimal
local group, i.e. $G_v \subseteq G_x$, for all $x \in \sigma$. Let
  \begin{equation} \label{Equ:weight-orbifold}
     \mathbf{w}(\sigma) = \mathrm{min}\{ |G_{v}|\,;\, v
      \ \text{is a vertex of}\  \sigma \}.
  \end{equation} 
  
    \begin{lem} \label{Lem:vertex_order}
  For any simplex $\sigma$ in a triangulation $\mathcal{T}$ adapted to $\mathcal{M}$, we can order the vertices of $\sigma$ 
   to be $\{v_0,\cdots,v_k\}$ so that 
   $\mathbf{w}(\sigma)=\mathbf{w}(v_0) \mid  \cdots \mid \mathbf{w}(v_k)$. So for any face $\sigma'$ of 
  $\sigma$ in $\mathcal{T}$, we have $\mathbf{w}(\sigma) \mid \mathbf{w}(\sigma')$.
   \end{lem} 
   \begin{proof}
   If the local group of $\mathcal{M}$ on the simplex $\sigma$ is constant, the lemma clearly holds.   
    Otherwise, there exists a proper face $\tau$ of $\sigma$ such that
     the local group is constant on $\sigma\backslash \tau$, and larger on $\tau$.  Let $\{v_0,\cdots, v_{s-1}\}$ be all the vertices of $\sigma\backslash \tau$ and $\{v_{s},\cdots, v_k\}$ be all the vertices of $\tau$. If the local group is constant on $\tau$, then 
     we have 
     $G_{v_0}=\cdots = G_{v_{s-1}}\subsetneq G_{v_s} = \cdots = G_{v_k}$.     
     Otherwise,  there exists a proper face $\tau'$ of $\tau$ such that
     the local group is constant on $\tau\backslash \tau'$, and larger on $\tau'$. By reordering the vertices of $\tau$ if necessary, we can assume that $v_{s},\cdots, v_{s'}$ are all the vertices of $\tau\backslash \tau'$ where $s \leq s'< k$. So
      $G_{v_0}=\cdots = G_{v_{s-1}}\subsetneq G_{v_{s}} = \cdots = G_{v_{s'}}$. By iterating the above argument, we can order the vertices of $\sigma$ to be 
      $\{v_0,\cdots,v_k\}$ so that $G_{v_0}\subseteq \cdots \subseteq G_{v_k}$. Then we have 
      $\mathbf{w}(\sigma)=\mathbf{w}(v_0) \mid  \cdots \mid \mathbf{w}(v_k)$ since $v_0$ is the vertex with the minimal local group. \n
      
      If $\sigma'$ is a face of $\sigma$, then the vertex set of $\sigma'$ is a subset of $\{v_0,\cdots,v_k\}$, say
      $\{v_{i_0},\cdots, v_{i_s}\}$, $0\leq i_0 <\cdots < i_s \leq k$. Then
      $ \mathbf{w}(\sigma)=\mathbf{w}(v_0) \mid \mathbf{w}(v_{i_0})=\mathbf{w}(\sigma')$.       
   \end{proof}
   \n

   \begin{prop} \label{Prop:Weight-Equal}
  For any orbifold $\mathcal{M} = (M,\mathcal{U})$, we have 
   $|G_x| = \mathbf{w}(\mathrm{Car}_{\mathcal{T}}(x))$ for each point $x\in M$.
   So $(M,\lambda_{\mathcal{M}})$ is a weighted polyhedron of descending type. 
   \end{prop}
\begin{proof}
 Take an adapted triangulation $\mathcal{T}$ of $\mathcal{M}$. 
   By Lemma~\ref{Lem:vertex_order}, $\mathbf{w}$ is a divisible weight of descending type on $\mathcal{T}$.
   For a point $x\in M$, let the vertex set of 
 $\mathrm{Car}_{\mathcal{T}}(x)$ be $\{v_0,\cdots, v_k\}$
 where $\mathbf{w}(v_0)\, |\, \cdots\, | \, \mathbf{w}(v_k)$. If the local group is constant on $\mathrm{Car}_{\mathcal{T}}(x)$, the lemma clearly holds.
 Otherwise, by the property ($\scalebox{1.3}{$\star$}$) of $\mathcal{T}$, there exists a proper face $\tau$ of 
 $\mathrm{Car}_{\mathcal{T}}(x)$ such that
 the local group is constant on $\mathrm{Car}_{\mathcal{T}}(x)\backslash \tau$. This implies $v_0\notin \tau$. On the other hand,  since $x$ is in the relative interior of $\mathrm{Car}_{\mathcal{T}}(x)$, we also have
 $x\notin \tau$. Hence $|G_x| = \mathbf{w}(v_0)= \mathbf{w}(\mathrm{Car}_{\mathcal{T}}(x))$ by the definition of $\mathbf{w}$ in~\eqref{Equ:weight-orbifold}.
 This implies that $(\mathcal{T},\mathbf{w})$ is a divisibly weighted triangulation of  $(M,\lambda_{\mathcal{M}})$.
\end{proof}

\subsection{$ws$-singular cohomology}
\ \n
For an orbifold $\mathcal{M} = (M,\mathcal{U})$, the space  $M$ carries a natural
stratification whose strata are the connected components of the sets
 $$\Sigma_H (M) := \{ x\in M\,|\, (G_x)=(H)\},$$
  where $H$ is any finite subgroup of $\mathrm{O}(n,\R)$
  and $(H)$ is its conjugacy class.\n

In~\cite[Section 2]{Yok07}, Takeuchi and Yokoyama introduce the following notion. An $n$-dimensional \emph{$s$-singular simplex} of $\mathcal{M}$ is a continuous map $\theta: |\Delta^n| \to M$ such that, for each $k$-face $\sigma$ of $\Delta^n$, $0\leq k \leq n$:
 \begin{enumerate}
     \item $\theta|_{|\sigma|^{\circ}}$ is contained in a single stratum of $\mathcal{M}$;
     \item there exists at most one $(k-1)$-face $\tau$ of $\sigma$ such that $\theta|_{|\tau|^{\circ}}$ lies in a stratum with higher order local group. 
 \end{enumerate}
 Since by definition $\theta(|\sigma|^{\circ})$ is contained in a single stratum of $\mathcal{M}$, the value of the weight function $\lambda_{\mathcal{M}}$ is constant on $\theta(|\sigma|^{\circ})$. Then it is meaningful to define a weight function $\mu_{\theta}$ on $\Delta^n$ by: for any face $\sigma$ of $\Delta^n$,
$$ \mu_{\theta}(\sigma) := \lambda_{\mathcal{M}} (\theta(x)) \ \text{where $x\in |\sigma|^{\circ}$}.$$
The weight $\mu_{\theta}$ on $\Delta^n$ is divisible and is of descending type because of the nature of $\lambda_{\mathcal{M}}$. Then we obtain a
weight-preserving singular $n$-simplex
 $$ \widetilde{\theta} : (|\Delta^n|,\lambda_{\mu_{\theta}}) \rightarrow (M,\lambda_{\mathcal{M}}).$$\n

In~\cite{Yok07}, the $s$-singular chain complex $s\text{-}C_*(\mathcal{M})$ is defined to be the free abelian group
generated by all the $s$-singular simplices with
the boundary operator $\partial$ defined in the same manner as the ordinary singular chain complex.
If we identify any $s$-singular simplex $\theta$ with its associated weight-preserving singular simplex $\widetilde{\theta}$,
the chain groups $\mathcal{WP}_n(M, \lambda_{\mathcal{M}})$ and $s\text{-}C_n(\mathcal{M})$ are isomorphic. 
 But the boundary operator $\partial$ on $s\text{-}C_*(\mathcal{M})$ does not involve the weight $\lambda_{\mathcal{M}}$ while the boundary operator $\partial^{\lambda_{\mathcal{M}}}$ on $\mathcal{WP}_*(M, \lambda_{\mathcal{M}})$ does. This makes their homology groups different.\n

Let $(s\text{-}C^*(\mathcal{M}),\delta)$ be cochain complex associated to $(s\text{-}C_*(\mathcal{M}),\partial)$. Moreover, \cite[Section 3]{Yok07} introduces 
a special sub-cochain complex of $s\text{-}C^*(\mathcal{M})$ as follows:
\begin{align*}
    ws \text{-} C^n(M)  :=  \{\phi\in s\text{-}C^n(\mathcal{M}) \mid \phi(\theta) \in \mu_{\theta}(\Delta^n) \Z \ & \text{for any  $s$-simplex} \ \theta: |\Delta^n| \rightarrow M \}.
\end{align*}
The cohomology of $(ws\text{-}C^*(\mathcal{M}),\delta)$ is called the \emph{$ws$-singular cohomology} of $\mathcal{M}$.\n
On the other hand, from the perspective of weighted singular cohomology we can apply the $\mathrm{Hom}(-, \Z)$ functor to
$(\mathcal{WP}_*(M, \lambda_{\mathcal{M}} ), \partial^{\lambda_{\mathcal{M}}})$ and obtain a cochain complex, denoted by $(\mathcal{WP}^*(M, \lambda_{\mathcal{M}}), \delta^{\lambda_{\mathcal{M}}})$. It follows from 
Theorem~\ref{thm:iso-weight-preserving} that the cohomology of $(\mathcal{WP}^*(M, \lambda_{\mathcal{M}}),\delta^{\lambda_{\mathcal{M}}})$, denoted by $\mathcal{H}_{WP}^*(M, \lambda_{\mathcal{M}})$, is isomorphic to $\mathcal{H}_{W}^*(M, \lambda_{\mathcal{M}})$. Moreover,
we can define a linear map 
 $$\Pi: (\mathcal{WP}^*(M, \lambda_{\mathcal{M}}), \delta^{\lambda_{\mathcal{M}}}) \rightarrow (ws \text{-}C^*(\mathcal{M}), \delta)$$
 where for any $\phi \in \mathcal{WP}^n(M, \lambda_{\mathcal{M}})$ and any $s$-singular simplex $\theta: |\Delta^n| \to M$,
\begin{equation*}
    \Pi(\phi)(\theta):= \mu_{\theta}(\Delta^n) \phi(\widetilde{\theta}). 
\end{equation*}
Clearly $\Pi(\phi) \in ws \text{-} C^n(M)$. It is also easy to verify that
$ \Pi(\delta^{\lambda_{\mathcal{M}}} \phi) = \delta (\Pi(\phi))$.
So the map $\Pi$ is a chain isomorphism.
 This proves the following theorem.
 
\begin{thm} \label{thm:ws-isom}
    For any orbifold $\mathcal{M}=(M, \mathcal{U})$,
    there is an isomorphism 
    \begin{equation}
        \mathcal{H}_W^*(M, \lambda_{\mathcal{M}}) \cong ws \text{-} H^*(\mathcal{M}). 
    \end{equation}
\end{thm}
\n

Later, some explicit examples are computed to demonstrate the isomorphism in Theorem~\ref{thm:ws-isom} (see Example~\ref{Exam:surface-sing} and Remark~\ref{Rem:Check}).

\subsection{Generalized equivariant homology}
\ \n
For an orbifold $\mathcal{M}$, one can also obtain a homology theory by considering the homology group of the \textit{classifying space} $B\mathcal{M}$ (see \cite[Definition 1.56]{AdemLeiRuan07}). An orbifold $\mathcal{M}$ is called a \textit{good orbifold} if it is isomorphic to a global quotient of a smooth manifold $Y$ by the action of a discrete group $G$. 
In this case, the classifying space of $\mathcal{M}$ is given by the Borel construction $EG\times_G Y$ (see \cite[Proposition 1.51]{AdemLeiRuan07}). The homology group of $EG\times_G Y$ is also known as the \textit{Borel equivariant homology group} of the $G$-space $Y$. Such a homology theory may be viewed as a special case of a \textit{generalized equivariant homology theory}. In this subsection, we show that descending type weight homology is likewise a special case of generalized equivariant homology theory. For simplicity, we omit the term generalized and refer to such a theory simply as equivariant homology.

Parallelly to the ordinary homology theory, there exist 
both singular version and simplicial version of equivariant homology.
In \cite{Illman73}, Illman introduced the concept of \textit{equivariant singular (co)homology}. In \cite{Bredon67}, Bredon introduced a simplicial version of
the equivariant cohomology. Although Bredon did not formally introduce equivariant simplicial homology, it can be defined analogously. In \cite{Han11}, Hanson explicitly defined \textit{equivariant simplicial homology} and showed that it is equivalent to the equivariant singular homology defined in~\cite{Illman73}. Since
it is more direct to build the connection between weighted homology and equivariant homology, we present only the theory of equivariant simplicial homology in detail.
Most of the results in this subsection about equivariant homology are cited from \cite{Han08} and \cite{Han11}.

\begin{rem}\label{rem:finiteness}
We require $Y$ to be a manifold in the first paragraph above in order to obtain an orbifold from the global quotient. However, for the study of equivariant homology theory in the following, it suffices to assume that $G$ is a discrete group and that $Y$ is a topological space such that the isotropy group at each point is finite.
\end{rem}

The theory of equivariant simplicial homology is closely related to the notions of simplicial $G$-complex and equivariant triangulation. Roughly speaking, a simplicial complex $L$ is called a \textit{simplicial $G$-complex} if $L$ admits an action of the group $G$ such that: for each $n$-simplex $\sigma \in L$ and each $g\in G$, the simplex $g \sigma$ is defined and is also an $n$-simplex of $L$ (see \cite[Definition 3.4]{Han08}). We do not introduce this notion in full detail, since the abstract version defined below is more convenient for the development of equivariant simplicial homology.

\begin{defi}
    Let $V$ be a $G$–set. An \textit{abstract simplicial $G$–complex} is
a collection of finite nonempty subsets $L$ of $V$ satisfying (i) if $v\in V$, then ${v} \in L $, (ii) if $\sigma\in L$ and $\tau \subseteq \sigma$, then $\tau \in L$, (iii) if $\sigma \in L$, then $g\sigma\in L$ for all $g \in G$, and (iv) if $\sigma\in L$ and $g \in G_\sigma$, then $gv = v$ for all $v \in \sigma$. Here, the notation $G_{\sigma}$ denotes the subgroup of $G$ that preserves the subset $\sigma \subseteq V$.
\end{defi}

\begin{rem}
    The last condition in the above definition excludes the situation where some element $g\in G_{\sigma}$ permutes the vertex set of $\sigma$ nontrivially.
\end{rem}

For an abstract $G$-simplicial complex $L$, one can define the action of $G$ on its realization $|L|$, thereby making $|L|$ a $G$-space (see \cite[Section 4]{Han08}).
Then $L$ is called an \emph{equivariant triangulation} of $Y$ if $|L|$ is $G$-equivariantly homeomorphic to $Y$. A natural question concerning equivariant triangulations is under what conditions a $G$-space admits an equivariant triangulation. The main theorem in \cite{Illman83} provides an answer:

\begin{thm}[{\cite{Illman83}}]
    Let $M$ be a smooth $G$-manifold with or without boundary. Then there exists an equivariant triangulation of $M$.
\end{thm}

\begin{rem}
    Here we adopt the definition of equivariant triangulation given by Hanson in \cite{Han08}. Although it differs literally from that given by Illman in \cite{Illman83}, the comment in the last section of \cite{Han08} indicates that these two definitions are actually equivalent.
\end{rem}

Now, we begin to introduce the concept of the \textit{equivariant simplicial chain complex}. 
Let $\mathcal{O}(G)$ denote the category whose objects are the left cosets of $G$. There are two kinds of morphisms in this category:
\begin{itemize}
    \item For $t\in G$ and $H\leq G$, define $\mu(t, H): G/H^t \to G/H$ by $gH^{t} \mapsto gtH$. Here, $H^t:=tHt^{-1}$.
    \item For $H_1\leq H_2 \leq G$, define $\kappa(H_2, H_1): G/H_1 \to G/H_2$ by $gH_1\mapsto gH_2$. 
\end{itemize}
Let $R$ be a ring. 
The definition of an equivariant simplicial complex relies on \textit{a covariant $R$–coefficient system} for $G$, namely a covariant functor $$\mathbf{m}: \mathcal{O}(G)\to R\text{-}mod.$$

\begin{rem}
    In the category $\mathcal{O}(G)$, the morphism $\mu(t, H)$ corresponds to the $G$-action on an orbit of $M$ that translates a point $x\in M$ to $t\cdot x$.
\end{rem}

\begin{defi}
    Let $L$ be an abstract simplicial $G$-complex, and let $\mathbf{m}$ be a covariant $R$-coefficient system for $G$. For each $\sigma \in L$, 
    \begin{equation}\label{eq:m[s]}
        \mathbf{m}[\sigma]:= \Bigl( \bigoplus_{gG_{\sigma} \in G/G_{\sigma}} \mathbf{m}(G/G^g_{\sigma}) \otimes_R R\langle g\sigma \rangle \Bigr)/ J_{\mathbf{m}}[\sigma].
    \end{equation}
    Here, for $g\in G$, $G_{\sigma}^g$ denotes the subgroup $g G_{\sigma} g^{-1}$, $R\langle g\sigma \rangle$ denotes the free $R$-module generated by $g\sigma\in L$, and $J_{\mathbf{m}}[\sigma]$ denotes the submodule generated by all elements of the form 
    \begin{equation}\label{eq:module-relation}
        x\otimes g\sigma- \mathbf{m}\bigl(\mu(g,G_{\sigma})\bigr) (x)  \otimes \sigma, 
    \end{equation}
    where $x \in \mathbf{m}(G/G^g_{\sigma})$. We use square brackets $[x\otimes g\sigma]$ to denote the elements of $\mathbf{m}[\sigma]$ represented by $x\otimes g\sigma$.

    Let $L_n$ denote the set of $n$-simplices of $L$. We define the $n$-th \textit{equivariant simplicial chain group} of $L$ by 
    \begin{equation}\label{eq:eqi-chain-group}
        C_n(L; \mathbf{m}):= \bigoplus_{[\sigma] \subseteq L_n} \mathbf{m}[\sigma],
    \end{equation}
    where the direct sum ranges over the orbits $[\sigma]$ in $L_n$.
\end{defi}

\begin{rem}
    The $n$-th (ordinary) simplicial chain group of $L$ with coefficients in $R$ can be written as
    \begin{equation*}
        C_n(L;R)=\bigoplus_{\sigma \in L_n} R\langle \sigma \rangle
    \end{equation*}
    where each $n$-simplex generates a direct summand. 
    In the equivariant simplicial chain complex, however, not every simplex generates a direct summand. Rather, the $R$-modules of the form in \eqref{eq:m[s]} generated by simplices in the same orbit of $L_n$ are identified by the relation given in \eqref{eq:module-relation}.
\end{rem}

As usual, we assume that the vertex set $V(L)$ of $L$ is equipped with a total order, so that we may define the $j$-th face operator $\partial_j$ (see the definition below \eqref{Equ:weighted-Boun-AW}). In addition, we do not require such an order on $V(L)$ being preserved by the $G$-action.

\begin{defi}
 Define
    $\partial^G : C_n(L; \mathbf{m}) \to C_{n-1}(L; \mathbf{m})$
by extending 
\begin{equation}\label{eq:m(partial^G)}
    \partial^G([x \otimes g \sigma]) := \sum_{j=0}^n (-1)^j [\mathbf{m} \bigl(\kappa(G_{\partial_j\sigma}^g,G_{\sigma}^g)\bigr) (x) \otimes g(\partial_j \sigma)].
\end{equation}
 linearly, where
$\sigma \in L_n$ and $x\in \mathbf{m}(G/G_{\sigma})$.
\end{defi}
To verify that $\partial^G$ is well defined, it suffices to check that
\begin{equation*}
    \partial^G([x\otimes g\sigma]) = \partial^G ([\mathbf{m}\bigl(\mu(g, G_{\sigma})\bigr)(x) \otimes \sigma]),
\end{equation*}
which follows directly from the definition.

Let $R = \mathbb{Z}$. We define a functor $\mathbf{m}_{\Z}$ by assigning the group $\mathbb{Z}$ to every $G$-orbit, the identity map on $\Z$ to each $\mu(t, H)$, and multiplication by $|H_2/H_1|$ to each $\kappa(H_2, H_1)$ whenever the coset $H_2/H_1$ is finite. If the coset $H_2/H_1$ is infinite, we define $\mathbf{m}_{\Z}(\kappa(H_2, H_1))=0$. In the following, we focus on the case where $H_2$ and $H_1$ are both finite subgroups. Then we have
\begin{equation*}
    \mathbf{m}_{\Z}(\kappa(H_2, H_1)) = |H_2/H_1| = |H_2|/|H_1|.
\end{equation*}

To study the quotient of $L$ under the $G$-action within the category of simplicial complexes, we require the simplicial $G$-complex $L$ to be a \textit{regular $G$-complex} (see \cite[Chapter III, Section 1]{Bredon72}). Here we list some basic facts on regular $G$-complexes.
The reader is referred to~\cite{Bredon72} for more detailed discussion on this subject.
\begin{enumerate}
    \item[(a)] The second barycentric subdivision of any abstract simplicial $G$-complex is a regular $G$-complex;\n
    
    \item[(b)] The quotient of a regular $G$-complex under the $G$-action is an abstract simplicial complex;\n
    
    \item[(c)] In a regular $G$-complex, no two vertices of a simplex can lie in the same $G$-orbit. Consequently, the image of each $n$-simplex under the quotient map is also an $n$-simplex.
\end{enumerate}

Suppose $L$ is a regular $G$-complex and $K=L\slash G$ is the quotient of $L$ by the $G$-action. According to the fact (b) listed above, $K$ is a simplicial complex. Let $\pi: L\rightarrow K$ denote the quotient map which is clearly a simplicial map.
Define a weight function $\xi$ on $K$ by setting $\xi(\tau)$ to be the order of the isotropy group of any simplex in $\pi^{-1}(\tau)$. So for any simplex $\sigma$ of $L$, $\xi(\pi(\sigma))$ is the order of the local group $G_{\sigma}$ of $\sigma$.
It is easy to see that $(K, \xi)$ is a descending-type weighted simplicial complex. \n

Let $V(K)$ and $V(L)$ denote the vertex set of $K$ and $L$, respectively. Then from a total order on $V(K)$, we can define a total order on $V(L)$ so that the quotient map
$\pi$ is order-preserving, i.e. if $v<v'$ in $V(L)$, then $\pi(v)\leq \pi(v')$ in $V(K)$. Taking these orders into account, we can deduce from the above fact (c) that
\begin{equation}\label{eq:pi-com-jface}
    \partial_j \circ \pi = \pi \circ \partial_j, \quad 0\leq j \leq n.
\end{equation}

\begin{lem}
    Keeping the notations as above, we have a chain isomorphism:
    \begin{equation}
        (C_*(L; \mathbf{m}_{\Z}), \partial^G) \cong (C_*(K, \xi), \partial^{\xi}). 
    \end{equation}
\end{lem}

\begin{proof}
    Note that the chain group $C_n(K, \xi)$ can be written as $\bigoplus_{\tau \in K_n} \Z \otimes_{\Z} \Z\langle \tau \rangle$. 
    Define $\pi_{\sharp}: C_n(L; \mathbf{m}_{\Z}) \to C_n(K; \Z)$ by: for any $n$-simplex $\sigma$ of  $L$,
    \begin{equation*}
        \pi_{\sharp}([k\otimes \sigma]) := k \otimes \pi(\sigma). 
    \end{equation*}
    Here, we need $L$ to be regular so that $\pi(\sigma)$ is an $n$-simplex of $K$. 
    It is clear that $\pi_{\sharp}$ is a well defined linear isomorphism. 

    Next, we verify that $\pi_{\sharp}\circ \partial^G ( [k\otimes \sigma])  = \partial^{\xi}\circ \pi_{\sharp} ([k\otimes \sigma])$ for any $n$-simplex $\sigma$ of $L$. 
    By \eqref{eq:m(partial^G)} and the definition of $\mathbf{m}_{\Z}$, we have 
    \begin{equation}  \label{Equ:pd-1}
        \pi_{\sharp}\bigl( \partial^G( [k\otimes \sigma ] ) \bigr) = \sum_j (-1)^j \frac{|G_{\partial_j \sigma}|}{|G_{\sigma}|}k \otimes \pi( \partial_j \sigma) . 
    \end{equation}
    On the other hand, we have 
    \begin{equation} \label{Equ:pd-2}
        \partial^{\xi}( \pi_{\sharp}([k\otimes \sigma]) )
        = \sum (-1)^j k \otimes \frac{\xi(\partial_j \pi(\sigma))}{\xi(\pi(\sigma))} \partial_j \pi(\sigma). 
    \end{equation}
   By the definition of $\xi$, we have $\xi(\pi(\sigma))=|G_{\sigma}|$. In addition, according to \eqref{eq:pi-com-jface} we have $\pi( \partial_j \sigma) = \partial_j \pi(\sigma)$. So we obtain
        $\xi(\partial_j \pi(\sigma))= \xi(\pi(\partial_j \sigma)) = |G_{\partial_j \sigma}|$. Then from~\eqref{Equ:pd-1} and~\eqref{Equ:pd-2}, we immediately obtain that
 $\pi_{\sharp}\circ \partial^{G} = \partial^{\xi} \circ \pi_{\sharp}$. Therefore, $\pi_{\sharp}: (C_*(L; \mathbf{m}_{\Z}), \partial^G) \to (C_*(K, \xi), \partial^{\xi})$ is a chain isomorphism.
\end{proof}

According to the above lemma, for a regular $G$-complex
$L$ and $K=L\slash G$,
\begin{equation}\label{eq:iso-eqivariant-simplicial}
    H_*(L; \mathbf{m}_{\Z}) \cong H_*(K, \partial^{\xi}). 
\end{equation}
 Assume a $G$-space $Y$ admits an equivariant triangulation $L$. Then one can define the equivariant simplicial homology of $Y$ as that of the $G$-simplicial complex $L$. In \cite{Han11}, Hanson proved that the equivariant simplicial homology and the equivariant singular homology are isomorphic. 
If we let $X=Y/G$, then $(X, \lambda_{\xi})$ is a descending type weighted polyhedron. Since \Cref{Thm:Main-Isom} shows that the weighted simplicial homology and the weighted singular homology are isomorphic, 
 we can immediately deduce the following theorem  from the isomorphism in \eqref{eq:iso-eqivariant-simplicial}.
 \begin{thm}
Let $G$ be a discrete group acting on a space $Y$. Assume that $Y$ admits an equivariant triangulation and that the isotropy group at each point is finite. Let $X=Y/G$, and define a weight function $\lambda$ on $X$ to be the order of the local group at each point. Then we have
     \begin{equation*}
         H_*(Y; \mathbf{m}_{\Z}) \cong \mathcal{H}_*^W(X, \lambda). 
     \end{equation*}
 \end{thm}

 \vskip 0.4cm

  \section{Examples} \label{Sec:Example}
    
 As we have shown in Section~\ref{Sec:Relation},
  an orbifold is naturally a weighted polyhedron of descending type. It suggests us to consider the
  descending type weighted polyhedra as the generalization of orbifolds. We define the following notion.
   
 \begin{defi}[Pseudo-orbifold]
 A weighted polyhedron $(X,\lambda)$ is called a \emph{pseudo-orbifold} if it is of descending type and
 $ X_{reg}=\{x \in X \mid \lambda(x)=1\}$ is a dense open subset
 of $X$. We call any point in $X_{reg}$ a \emph{regular point} of $(X,\lambda)$, and any point in $X \backslash X_{reg}$ a \emph{singular point}.
\end{defi}

\n

 In the following, we compute the weighted singular (or simplicial) homology of some orbifolds and pseudo-orbifolds in dimension one and two. 
  For brevity, we simply use $\mathcal{H}^{W}_*(\mathcal{M})$
to denote the weighted singular homology of an orbifold or pseudo-orbifold $\mathcal{M}$ below.
 The excision and Mayer-Vietoris sequence of weighted simplicial homology (see~\cite[Theorem 2.3 and Corollary 2.3.1]{Daw90}) are the basic tools for the computation.\n
 
In addition, the definition of weighted simplicial homology makes perfect sense for any $\Delta$-complex (see~\cite{Hatcher02}) with a weight function.  
 So when we compute the weighted simplicial homology (or cohomology) of a weighted polyhedron $(X,\lambda)$ via a divisibly weighted triangulation $(K,\xi)$, 
  we may allow $K$ to be
  a $\Delta$-complex as long as $(K,\xi)$ consists of divisibly weighted simplices.
  We call such kind of $(K,\xi)$ a \emph{divisibly weighted $\Delta$-triangulation} of $(X,\lambda)$.
In many circumstances,  using $\Delta$-triangulations can significantly reduce the number of generators of a weighted simplicial chain complex and hence simplify the computation.

   \subsection{One-dimensional orbifolds and
        pseudo-orbifolds}\ \n
      
    \begin{exam}[One dimensional compact orbifolds]\label{Exam:1-dim-orbifold}
     It is well known that any one-dimensional compact connected orbifold is isomorphic to one of the following cases (see~\cite{Choi12}).     \begin{itemize}
      \item[(a)] $S^1$ or $[0,1]$ without singular points.\n
     \item[(b)] An interval $\mathcal{I}=[v_1,v_2]\subset \R^1$ with only one singular point $v_1$
     where $G_{v_1}= \Z_2$. \n
     
     \item[(c)] An interval 
     $\mathcal{I}'=[v_1,v_2]\subset \R^1$ with two singular points
     $v_1$, $v_2$ where $G_{v_1}=G_{v_2}=\Z_2$.
 
     \end{itemize}
     
     Note that $\mathcal{I}$ is a divisibly weighted $1$-simplex while $\mathcal{I}'$ is not. Using an adapted triangulation of $\mathcal{I}'$, we can easily obtain
    $$ \mathcal{H}^{W}_j(\mathcal{I})
      \cong \begin{cases}
   \Z,  &  \text{if $j=0$}; \\
   0,  &  \text{if $j\geq 1$}.
 \end{cases} \ \  \ \
\mathcal{H}^{W}_j (\mathcal{I}')
      \cong \begin{cases}
   \Z\oplus \Z\slash 2\Z ,  &  \text{if $j=0$}; \\
   0,  &  \text{if $j\geq 1$}.
 \end{cases}  $$
 
    \end{exam}

    There are much more one-dimensional pseudo-orbifolds based on an interval than orbifolds because the weights of
    the singular points in a pseudo-orbifold could be arbitrary positive integers. For example, let $\mathcal{I}_{(k_1,k_2)}$
    denote the pseudo-orbifold based on an interval
    $[v_1,v_2]$ (see Figure~\ref{Fig:Example-Interval-Sing}) where the endpoints $v_1$ and $v_2$ are the only singular points
    whose weights are $k_1$ and $k_2$, respectively.
      \begin{figure}[h]
        \begin{equation*}
        \vcenter{
            \hbox{
                  \mbox{$\includegraphics[width=0.37\textwidth]{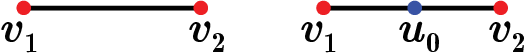}$}
                 }
           }
     \end{equation*}
   \caption{An interval with two singular points at the ends
       } \label{Fig:Example-Interval-Sing}
   \end{figure}  
   
     We can add a vertex $u_0$
    to the interval $[v_1,v_2]$ to obtain a divisibly weighted triangulation of $\mathcal{I}_{(k_1,k_2)}$ with two $1$-simplices $\overline{v_1u_0}$ and $\overline{u_0v_2}$. The corresponding weighted simplicial chain complex can be written as
    \[  \Z^2 \overset{\bordermatrix{%
	       & \overline{v_1u_0}   & \overline{u_0v_2}   \cr
	 v_1    &  -k_1       &  0     \cr
	  u_0    &  1         &  -1    \cr
	 v_2    &  0         & k_2     
    } }{\xlongrightarrow{\qquad \ \qquad\ \qquad\ \quad}} \Z^3. \]
   Then it is easy to compute
     $$ \mathcal{H}^{W}_j (\mathcal{I}_{(k_1,k_2)})
      \cong \begin{cases}
   \Z\oplus \Z\slash \mathrm{gcd}(k_1,k_2)\Z ,  &  \text{if $j=0$}; \\
   0,  &  \text{if $j\geq 1$}.
 \end{cases}  $$
 
Note that we can also consider $\mathcal{I}_{(k_1,k_2)}$ as a weighted simplex of descending type and compute its weighted simplicial homology directly. The answer is the same.
 
 \begin{exam}[Line segment with finitely many singular points in the interior]\label{Exam:Segment}
   \ \n

    Let $I_{(k_1,k_2,\dots,k_n)}$ denote 
    the pseudo-orbifold based on $[a,b]$ with $n$ singular points $v_1,\cdots, v_n$ where $ a<v_1<\cdots <v_n<b$ and
    the weight of $v_i$ is $k_i$, $1\leq i \leq n$ (see Figure~\ref{Fig:Line-Seg-2}).
   Let $(K,\mu)$ be the obvious weighted triangulation of $I_{(k_1,k_2,\dots,k_n)}$ where the vertex set of $K$ is   $\{a,b, v_1,\cdots, v_n\}$ and the $1$-simplices of $K$  are $\overline{av_1},  \overline{v_1v_2}, \cdots,
   \overline{v_{n-1}v_n}, \overline{v_nb}$.     
    The weighted simplicial chain complex
   of $(K,\mu)$ is 
     $$ C_1(K) \overset{\partial^{\mu}}{\longrightarrow} C_0(K) \longrightarrow 0 $$  
    where $\partial^{\mu}$ is represented by the following $(n+2)\times (n+1)$ matrix
    $$A= \bordermatrix{ 
          &  \overline{av_1}   & \overline{v_1v_2}  & \overline{v_2v_3} & \cdots & \overline{v_{n-1}v_n} & \overline{v_nb} \cr
	  a     &   - 1    &  0    &  0   & \cdots & 0 & 0 \cr
	 v_1    &   k_1    & -k_1  &  0   & \cdots & 0 & 0  \cr
	 v_2    &    0     & k_2   & -k_2 & \cdots & 0 & 0  \cr
	 v_3    &    0     &  0    &  k_3 & \cdots & 0 & 0  \cr 
	 \vdots & \vdots   &\vdots &\vdots &       &-k_{n-1} & 0  \cr
	 v_n    &    0     &  0    &  0    & \cdots & k_n &  -k_n \cr
	  b     &   0      &  0    &  0   & \cdots & 0 & 1 }.    $$
	   \begin{figure}[h]
   	\begin{equation*}
   		\vcenter{
   			\hbox{
   				\mbox{$\includegraphics[width=0.52\textwidth]{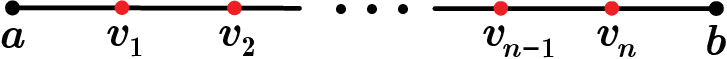}$}
   			}
   		}
   	\end{equation*}
   	\caption{An interval with $n$ singular point in the interior
   	} \label{Fig:Line-Seg-2}
   \end{figure} 
   
  By doing a sequence of elementary column transforms
  on $A$ from the right side to the left side followed by some elementary row transforms, we can simplify $A$ to the matrix $B$ as shown below:
  $$ A \ \longrightarrow \ \begin{pmatrix} 
           - 1    &  0    &  0   & \cdots & 0 & 0 \\
	    0    & -k_1  &  0   & \cdots & 0 & 0  \\
	      0     & 0   & -k_2 & \cdots & 0 & 0  \\
	    0     &  0    &  0 & \cdots & 0 & 0  \\ 
	  \vdots   &\vdots &\vdots &       & -k_{n-1} & 0  \\
	     0     &  0    &  0    & \cdots & 0 &  -k_n \\
	    1      &  1    &  1  & \cdots & 1 & 1 
	    \end{pmatrix} \ \longrightarrow \begin{pmatrix} 
           1    &  0    &  0   & \cdots & 0 & 0 \\
	    0    & k_1  &  0   & \cdots & 0 & 0  \\
	      0     & 0   & k_2 & \cdots & 0 & 0  \\
	    0     &  0    &  0 & \cdots & 0 & 0  \\ 
	  \vdots   &\vdots &\vdots &       & k_{n-1} & 0  \\
	     0     &  0    &  0    & \cdots & 0 &  k_n \\
	    0      &  1    &  1  & \cdots & 1 & 1 
	    \end{pmatrix} = B.$$
  Then it is easy to see that $ \mathcal{H}^{W}_1(I_{(k_1,\dots,k_n)})=0$. Moreover, in the associated weighted simplicial cochain complex (with integral coefficients)
    $$ C^1(K) \overset{\ \delta^{\mu}}{\longleftarrow} C^0(K) \longleftarrow 0,$$ 
    the weighted coboundary $\delta^{\mu}$ is represented by the transpose $A^{t}$ of $A$.  
  Then since $A^t\sim B^t$, we can easily see that $H^0(K,\delta^{\mu})\cong \Z$ and $H^1(K,\delta^{\mu})=\mathrm{coker}(\delta^{\mu})$ is isomorphic to the group $G_{\{k_1,\cdots,k_n\}}$ defined below: \begin{align}
      G_{\{k_1,\cdots, k_n\}} &:= \Z_{k_1}\langle g_1 \rangle \oplus \cdots \oplus \Z_{k_n}\langle g_n \rangle \slash \langle g_1+\ldots + g_n \rangle. 
    \end{align}  
   Finally, it follows from the universal coefficient theorem that
    \begin{equation} \label{Equ:DW-I-n}    
   \mathcal{H}^{W}_i(I_{(k_1,\dots,k_n)}) \cong H_i(K,\partial^{\mu}) \cong
        \begin{cases}
            \Z \oplus G_{\{k_1,\cdots,k_n\}}, & \text{if $i=0$;} \\ 
            0, & \text{otherwise}.
        \end{cases}
    \end{equation}
  
  Note that the order of $G_{\{k_1,\cdots,k_n\}}$ is 
   \begin{equation*} 
    |G_{\{k_1,\cdots,k_n\}}| =\frac{k_1k_2\cdots k_n}{\mathrm{lcm}(k_1,k_2,\cdots,k_n)}.
\end{equation*}
     In particular, $G_{\{k_1,\cdots,k_n\}}$ is trivial if and only if $n=1$ or $k_1,k_2,\cdots, k_n$ are pairwise relatively prime.\n
      
  We want to remark that $(K,\mu)$ may not be divisibly weighted. But we can add a vertex between $v_i$ and $v_{i+1}$ for each $1\leq i \leq n-1$ to obtain a divisibly weighted triangulation $(K',\mu')$ of $I_{(k_1,\dots,k_n)}$.
  It is easy to check that
 $H_*(K',\partial^{\mu'})$ is isomorphic to $H_*(K,\partial^{\mu})$. So the above calculation of $\mathcal{H}^{W}_i(I_{(k_1,\dots,k_n)}) $ using $(K,\mu)$ is valid.        
   \end{exam}

 \begin{exam}[Circle with discrete singular points]\label{Exam:1-dim-pseudo-orbifold}
       
  Let $\mathcal{S}^1_{(k_1,\cdots,k_n)}$ denote the $1$-dimensional pseudo-orbifold which is a circle with $n$ singular points $v_1,\cdots, v_n$ where the
  weight of $v_i$ is $k_i$, $1\leq i \leq n$.   We can think of $\mathcal{S}^1_{(k_1,\dots,k_n)}$ as the union of $I_{(k_1,\dots,k_n)}$ and $I=[a,b]$ (an interval without singular points) where $I_{(k_1,\dots,k_n)} \cap I=\{a,b\}$.
        The Mayer-Vietoris sequence of the weighted simplicial homology for the decomposition $\mathcal{S}^1_{(k_1,\dots,k_n)}=I_{(k_1,\dots,k_n)} \cup I$ gives
 \begin{align*}
 &  \mathcal{H}_0^{W}(\{a, b\}) \xlongrightarrow{i_*} \mathcal{H}_0^{W} \left(I_{\left(k_1, \ldots, k_n\right)}\right) \oplus \mathcal{H}_0^{W}(I) \longrightarrow \mathcal{H}_0^{W} \left(\mathcal{S}_{\left(k_1, \ldots, k_n\right)}^1\right) \longrightarrow 0 \\
 & \quad \ \ \Z\oplus \Z \ \ \ \longmapsto \ \ \ (\Z \oplus G_{\{k_1,\cdots,k_n\}}) \oplus \Z  
   \end{align*}
 Choose the generators of $\mathcal{H}_0^{W}(\{a,b\} )$ to be $[a]$ and $[a]-[b]$. The image of $[a]$ in $\mathcal{H}_0^{W} (I_{(k_1,\cdots,k_n)} )$ and $\mathcal{H}_0^{W} (I )$ are both generators of the free abelian part, while the image of $[a]-[b]$ is $0$. So  $\mathcal{H}_0^{W}(\mathcal{S}_{\left(k_1, \ldots, k_n\right)}^1) \cong \Z \oplus G_{\{k_1,\cdots,k_n\}}$. Then we obtain
    \begin{equation} \label{Equ:AW-S-n}
            \mathcal{H}^W_i(\mathcal{S}^1_{(k_1,\cdots,k_n)})=
            \begin{cases}
                \Z \oplus  G_{\{k_1,\cdots, k_n\}}, & \text{if $i=0$}; \\ 
                \Z, & \text{if $i=1$};\\
                0, & \text{otherwise}.
            \end{cases}
    \end{equation}

  \end{exam}
 
 \n
 
    \subsection{Two-dimensional orbifolds and
        pseudo-orbifolds}\ \n

   \begin{exam}\label{Exam:n-gon-Homology}
   In Figure~\ref{Fig:Example-n-gon}, we have a 
   $2$-dimensional pseudo-orbifold, denoted by $\mathcal{D}^2_{(k_1,\cdots,k_n)}$, which is a $2$-disk $D^2$ with only $n$ singular points $v_1,\cdots, v_n$ on the boundary where the weight of $v_i$ is $k_i$, $i=1,\cdots,n$. Then boundary of $\mathcal{D}^2_{(k_1,\cdots,k_n)}$ is just $\mathcal{S}^1_{(k_1,\cdots,k_n)}$. \n  
   \begin{figure}[h]
        \begin{equation*}
        \vcenter{
            \hbox{
                  \mbox{$\includegraphics[width=0.58\textwidth]{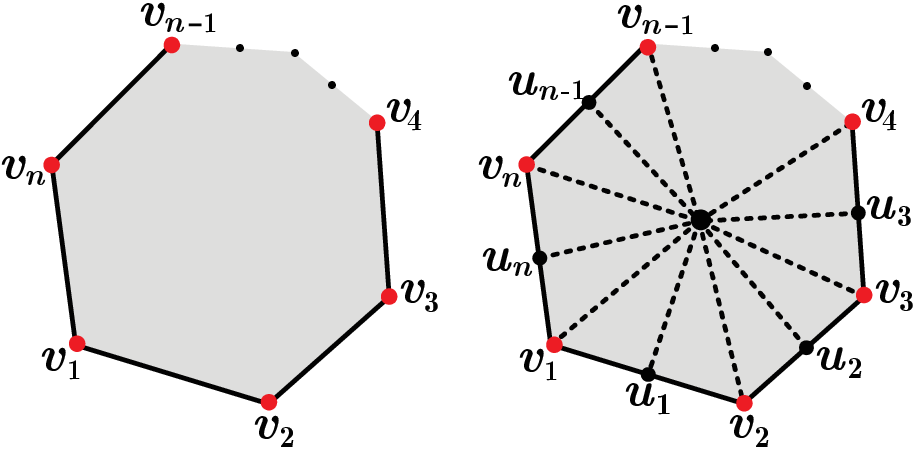}$}
                 }
           }
     \end{equation*}
   \caption{A $2$-disk with $n$ singular points on the boundary
       } \label{Fig:Example-n-gon}
   \end{figure}  
 Let $(K,\xi)$ denote the divisibly weighted triangulation of $\mathcal{D}^2_{(k_1,\cdots,k_n)}$ as shown in Figure~\ref{Fig:Example-n-gon}. A direct calculation shows that 
 \[ \mathcal{H}^{W}_j(\mathcal{D}^2_{(k_1,\cdots,k_n)})  \cong
       \begin{cases}
     \Z\oplus G_{\{k_1,\cdots,k_n\}},  &  \text{if $j=0$}; \\
     0,   &    \text{if $j\geq 1$}.
 \end{cases} 
    \]  
Moreover, a simple analysis of the weighted simplicial chain complexes of  $\mathcal{S}^1_{(k_1,\cdots,k_n)}$ and 
$\mathcal{D}^2_{(k_1,\cdots,k_n)}$ shows that the inclusion map
  $\mathcal{S}^1_{(k_1,\cdots,k_n)} \hookrightarrow 
  \mathcal{D}^2_{(k_1,\cdots,k_n)}$ induces an isomorphism $ \mathcal{H}^{W}_0(\mathcal{S}^1_{(k_1,\cdots,k_n)}) 
  \xlongrightarrow{\cong} \mathcal{H}^{W}_0(\mathcal{D}^2_{(k_1,\cdots,k_n)})$.    
  \end{exam}

  \begin{exam} \label{Exam:Teardrop}
    The \emph{teardrop} orbifold $\mathcal{S}^2_{(k)}$ is the $2$-sphere with only one singular point $v_0$
    (see Figure~\ref{Fig:Example-Teardrop})
whose local group $G_{v_0}=\Z\slash k\Z$ ($k\geq 2$). A chart around $v_0$ consists of an open disk $\tilde{U}\subset \R^n$ and
an action of $\Z\slash k\Z$ on $\tilde{U}$ by rotations.
$\mathcal{S}^2_{(k)}$ is a typical example of ``bad'' orbifold. Using the triangulation adapted to $\mathcal{S}^2_{(k)}$ given in Figure~\ref{Fig:Example-Teardrop}, we can compute 
    \[  
      \mathcal{H}^{W}_j(\mathcal{S}^2_{(k)}) \cong
       \begin{cases}
   \Z,  &  \text{if $j=0$}; \\
     0 ,   &    \text{if $j=1$}; \\
   \Z,  &  \text{if $j = 2$};\\
   0, & \text{if $j \geq 3$}.
 \end{cases} 
    \]    
    \begin{figure}[h]
        \begin{equation*}
        \vcenter{
            \hbox{
                  \mbox{$\includegraphics[width=0.35\textwidth]{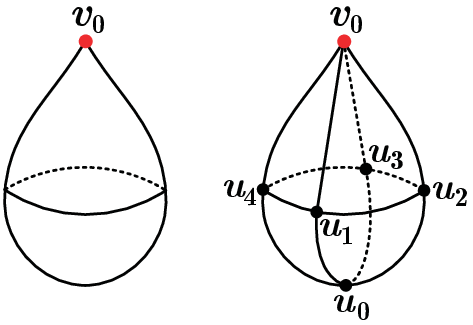}$}
                 }
           }
     \end{equation*}
   \caption{Teardrop orbifold
       } \label{Fig:Example-Teardrop}
   \end{figure} 
   
   \end{exam}
   
  \begin{exam}\label{Exam:Sphere-n-points}
  Let $\mathcal{S}^2_{(k_1,\cdots,k_n)}$ denote a pseudo-orbifold based on a $2$-dimensional sphere $S^2$ with only $n$ singular points $v_1,\cdots, v_n$ where the weight of $v_i$ is $k_i$, $1\leq i \leq n$. 
 We can think of $\mathcal{S}^2_{(k_1,\cdots,k_n)}$ as the gluing of two copies of $\mathcal{D}^2_{(k_1,\cdots,k_n)}$ along their boundaries, that is
 \[  \mathcal{S}^2_{(k_1,\cdots,k_n)} = \mathcal{D}^2_{(k_1,\cdots,k_n)} \cup_{\mathcal{S}^1_{(k_1,\cdots,k_n)}} \mathcal{D}^2_{(k_1,\cdots,k_n)}. \]
 Then by Mayer-Vietoris sequence,
 we can easily compute
    \[ \mathcal{H}^{W}_j(\mathcal{S}^2_{(k_1,\cdots,k_n)})  \cong
       \begin{cases}
   \Z\oplus G_{\{k_1,\cdots,k_n\}},  &  \text{if $j=0$}; \\
  0,   &    \text{if $j=1$}; \\
   \Z,  &  \text{if $j= 2$};\\
   0, & \text{if $j\geq 3$}.
 \end{cases} 
    \]

\end{exam}

  More generally, we can compute the weighted singular (or simplicial) homology of any pseudo-orbifold with finitely many singular points whose underlying space is a closed surface.
      
    \begin{exam} \label{Exam:surface-sing}
   Let $\mathcal{X}_{(k_1,\cdots,k_n)}$ denote a 
   pseudo-orbifold based on a connected surface
   $\Sigma$ with only $n$ singular points $v_1,\cdots, v_n$ where the weight of $v_i$ is $k_i$, $1\leq i \leq n$. By thinking of $\mathcal{X}_{(k_1,\cdots,k_n)}$ as the connected sum of 
  $\mathcal{S}^2_{(k_1,\cdots,k_n)}$ with the surface $\Sigma$ (with no singular points) along a regular circle, we can easily compute the weighted singular homology of $\mathcal{X}_{(k_1,\cdots,k_n)}$ from the preceding results on $\mathcal{S}^2_{(k_1,\cdots,k_n)}$ using 
  Mayer-Vietoris sequence:
   \[  \mathcal{H}^{W}_j(\mathcal{X}_{(k_1,\cdots,k_n)})\cong \begin{cases}
   \Z \oplus G_{\{k_1,\cdots,k_n\}} ,  & \text{if $j=0$}; \\
   H_1(\Sigma),  &  \text{if $j=1$}; \\
     \Z ,  &  \text{if $j=2$}; \\
    0 ,  &  \text{if $j\geq 3$}. \\
 \end{cases} \]
    where $H_1(\Sigma)$ is the ordinary singular homology group of $\Sigma$ in dimension $1$.   
    \end{exam}

    \begin{rem} \label{Rem:Check}
     Using the universal coefficient theorem, we can
     easily check that the weighted singular cohomology $ \mathcal{H}_{W}^*(\mathcal{X}_{(k_1,\cdots,k_n)})$ agrees with the $ws$-cohomology of $\mathcal{X}_{(k_1,\cdots,k_n)}$ computed in~\cite[Theorem 9.1]{TakYok07}.
    \end{rem}

   Our final example is to show the difference between the weighted simplicial cohomology and the ordinary simplicial cohomology of a weighted polyhedron.

 \begin{exam} \label{Exam:DW-cohom-Surface}
   Let $\Sigma_2$ denote a closed orientable surface of genus $2$ which is glued from an octagon as shown in  Figure~\ref{Fig:Example-Sigma-2}. For any integer $k\geq 2$, let $\lambda_k$ be a weight function on $\Sigma_2$ which assigns $1$ to all the
   interior point of the octagon and assigns $k$ to
   those points from the edges on the boundary of the octagon. 
   Then the middle picture in Figure~\ref{Fig:Example-Sigma-2} is a divisibly weighted $\Delta$-triangulation of
   the pseudo-orbifold 
   $(\Sigma_2,\lambda_k)$, from which we can compute the 
   weighted simplicial cohomology of $(\Sigma_2,\lambda_k)$:
   \[ H^p_{W}(\Sigma_2,\lambda_k) \cong 
   H^p(\Sigma_2) \cong \begin{cases}
   \Z,  &  \text{$p=0,2$}; \\
   \Z\oplus\Z\oplus\Z\oplus\Z, & \text{$p=1$}; \\
   0,  &  \text{otherwise}.
 \end{cases} 
  \]

  But the weighted simplicial cohomology ring $\big( H^*_{W}(\Sigma_2,\lambda_k), \Cup \big)$ is not isomorphic as a graded ring to the ordinary simplicial cohomology ring
  $\big( H^*(\Sigma_2), \cup \big)$ of $\Sigma_2$. Indeed, let $\phi_1,\phi_2$, $\psi_1$ and $\psi_2$ denote the generators of
  $H^1_{W}(\Sigma_2,\lambda_k)$ that are dual to
  $a_1,a_2, b_1$ and $b_2$, respectively. More specifically, $\phi_i$ (or $\psi_i$) has the value $1$ 
  on $a_i$ (or $b_i$) and the value $k$ on the two edges
   meeting the dotted arc that intersects $a_i$ (or $b_i$) (see the right picture in Figure~\ref{Fig:Example-Sigma-2}), and has the value $0$ on all other edges.   Let $\zeta$ denote the generator of $H^2_{W}(\Sigma_2,\lambda_k)$ which has the value $1$ on all the $2$-simplices of $\Sigma_2$. Then we can directly compute from the $\Delta$-triangulation and the definition of $\Cup$ that the only nontrivial relations among the generators of $H^1_{W}(\Sigma_2,\lambda_k)$ and $H^2_{W}(\Sigma_2,\lambda_k)$ are:
  \[ \phi_1\Cup \psi_1 = \phi_2\Cup \psi_2 = k^2 \cdot \zeta. \] 
 So $\zeta$ is a multiplicative
  generator of $\big( H^*_{W}(\Sigma_2,\lambda_k), \Cup \big)$ since $k\geq 2$.
  Therefore, there is no graded ring isomorphism between $\big( H^*_{W}(\Sigma_2,\lambda_k), \Cup \big)$
 and $\big( H^*(\Sigma_2), \cup \big)$.      
 \end{exam}
 
 \begin{figure}[h]
        \begin{equation*}
        \vcenter{
            \hbox{
                  \mbox{$\includegraphics[width=0.9\textwidth]{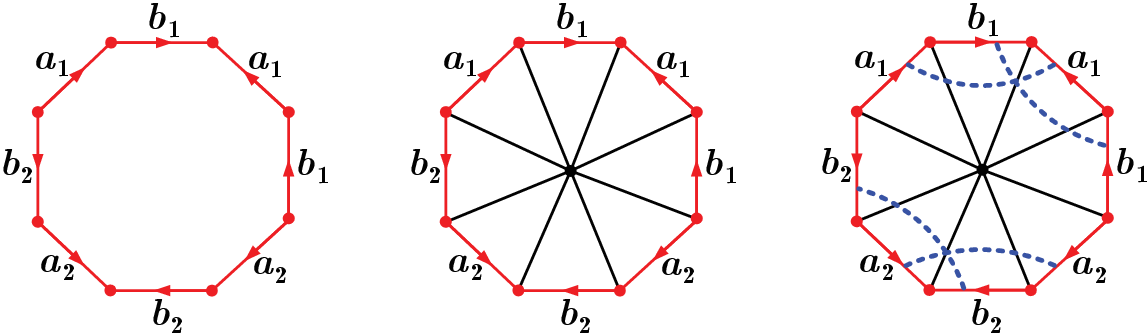}$}
                 }
           }
     \end{equation*}
   \caption{A pseudo-orbifold based on a genus $2$ closed orientable surface
       } \label{Fig:Example-Sigma-2}
   \end{figure}       
    
   \vskip .4cm

\section*{Appendix}

\textbf{\emph{Proof of Lemma~\ref{Lem:Inclusion-Equiv}}:}
 First, assume that $X$ is a convex set in an Euclidean space.
 let $L\mathcal{W}_*(X,\lambda)$ denote the subgroup of 
 \emph{linear chains} in $\mathcal{W}_*(X,\lambda)$, which is generated by all the weighted singular simplices $\theta: (|\Delta^{n}|,\lambda_{\xi}) \rightarrow (X,\lambda)$ with $\theta: |\Delta^n|\rightarrow X$ being a linear map. For any $n\geq 1$, let the vertex set of 
 $\Delta^n$ be $\{v_0,\cdots, v_n\}$.
 For a generator $\theta\in L\mathcal{W}_n(X,\lambda)$, \n
 \begin{itemize}
  \item For any point $p\in X$, we can define a linear map
 $\theta^p: |\Delta^{n+1}| \rightarrow X$ by 
   \begin{equation}\label{Equ:Cone-theta}
                     \theta^p (t_0,t_1,\ldots, t_{n+1}) =
                    \begin{cases}
                       \, p & \text{if\ } t_{0}=1; \\
                       t_{0} p + (1-t_{0})\, \theta (\frac{t_1}{1-t_{0}},\cdots,
                       \frac{t_{n+1}}{1-t_{0}}) & \text{if\ } t_{0}\neq 1.
                    \end{cases}
     \end{equation}
                where $(t_0,t_1,\ldots, t_{n+1})$ is the
                barycentric coordinates of $|\Delta^{n+1}|$. 
                Obviously, $\theta = \theta^p|_{\Delta^n}$ where
                $\Delta^n$ is identified with the face $[v_1,\cdots, v_{n+1}]$ of $\Delta^{n+1}$.\n
                
    \item  Let $b_{\Delta^n}$ be the barycenter of $\Delta^n$. Then the divisible weight $\xi$ on $\Delta^n$ canonically determines a divisible weight $\xi^*$ on $\Delta^{n+1}$ by
     \begin{equation} \label{Equ:xi-star}
     \qquad \xi^*(v_0)=
    \xi(\Delta^n)=Sd(\xi)(b_{\Delta^n}), \ \ 
    \xi^*|_{\Delta^n}=\xi.
    \end{equation}

  \item Let $p_{\theta}=\theta(b_{\Delta^n})$. Then $\theta^{p_{\theta}} :  (|\Delta^{n+1}|,\lambda_{\xi^*}) \rightarrow (X,\lambda)$ defines a linear weighted singular $(n+1)$-simplex in
  $L\mathcal{W}_{n+1}(X,\lambda)$.
   We also denote  $\theta^{p_{\theta}}$ by $p_{\theta}\cdot\theta $ and call it
  the \emph{cone} of $\theta$ with $p_{\theta}$. \n

 \item Generally, if a linear weighted singular $n$-simplex 
 $\theta: (|\Delta^{n}|,\lambda_{\xi}) \rightarrow (X,\lambda)$ in $L\mathcal{W}_n(X,\lambda) $ can be extended to $\theta': (|\Delta^{n+1}|, \lambda_{\xi'}) \rightarrow
 (X,\lambda)$ in $
 L\mathcal{W}_{n+1}(X,\lambda)$ with $\xi'(v_0)=\lambda(\theta'(v_0))$ (i.e. $\theta'$ preserves the weight of $v_0$), then we write
\begin{equation} \label{Equ:Cone-p}
 \theta' =p\cdot \theta, \ \text{where}\ p=\theta'(v_0).
 \end{equation}
We call $\theta'$ the \emph{cone} of $\theta$ with $b$.
 \end{itemize}
 \n
 
 Moreover, we consider the cone over a point $p\in X$ as a linear operation on the weighted singular chains of $(X,\lambda)$ whenever it makes sense. Of course, for a given linear weighted singular simplex $\theta$, not all points $p\in X$ can we take a cone with $\theta$. But as we have seen, $p_{\theta}=\theta(b_{\Delta^n})$ is always a valid cone point for $\theta$.
 \n
 \begin{itemize}

 \item For any weighted singular simplex $\theta: (|\Delta^{n}|,\lambda_{\xi}) \rightarrow (X,\lambda)$ in $L\mathcal{W}_n(X,\lambda) $,
 \begin{equation} \label{Equ:bound-lambda-star}
   \partial^{\lambda} ( p_{\theta}\cdot \theta )  = 
   \theta  - p_{\theta}\cdot \partial \theta,
  \end{equation}
  where $\partial: L\mathcal{W}_*(X,\lambda) \rightarrow
  L\mathcal{W}_*(X,\lambda) $ is defined by
\begin{equation} \label{Equ:bound-normal}
 \partial \theta:= \sum^n_{j=0}(-1)^j \cdot\theta|_{\partial_j\Delta^n}. 
 \end{equation}
 The equality in~\eqref{Equ:bound-lambda-star} follows from the fact that $\Delta^{n+1}$ and all its $n$-faces have the same weight value under $\xi^*$.  
Note that $\partial$ is also a valid boundary map on $L\mathcal{W}_*(X,\lambda)$
 which
 is different from $\partial^{\lambda}$ in~\eqref{Equ:Bd-Sing}.\n
  
 \end{itemize}
 
  Similarly to the ordinary chain map induced by the barycentric subdivision in~\eqref{Equ:Sd-simple}, we can inductively define a chain map $Sd_{\#}$ on $L\mathcal{W}_*(X,\lambda)$ as follows:
  for a linear weighted singular $n$-simplex $\theta: (|\Delta^{n}|,\lambda_{\xi}) \rightarrow (X,\lambda)$, let
 \begin{align} \label{Equ:Sd-Weight-Sing}
   Sd_{\#} (\theta)= \begin{cases}
  p_{\theta}\cdot Sd_{\#} (\partial\theta),  &  \text{if 
  $n\geq 1$}; \\
 \theta,  &  \text{if $n=0$}.
 \end{cases} 
    \end{align}

\n
  Next, we inductively construct
 a chain homotopy 
 $$T=\{T_n : L\mathcal{W}_n(X,\lambda) \rightarrow L\mathcal{W}_{n+1}(X,\lambda)\}$$ between $Sd_{\#}$ and the identity on $L\mathcal{W}_*(X,\lambda)$ by: $T_{-1}=0$ and
  \begin{align} \label{Equ:Def-Tn}
   T_n(\theta) & :=  p_{\theta}\cdot ( \theta- T_{n-1}(\partial\theta ) ).
   \end{align}
     Note that if we forget the weights, our constructions of $Sd_{\#}$ in~\eqref{Equ:Sd-simple} and $T$ in~\eqref{Equ:Def-Tn} are equivalent to the constructions 
   in~\cite[Proposition 2.21]{Hatcher02} for the ordinary singular chain complex. So by the same argument as in~\cite{Hatcher02}, we have    
   \begin{equation} \label{Equ:T-Relation-1}
    \partial \circ T + T \circ \partial = \mathrm{id} - Sd_{\#}.
    \end{equation}   
  But here we need to prove
   $\partial^{\lambda} \circ T + T \circ \partial^{\lambda} = \mathrm{id} - Sd_{\#}$ on
   $(L\mathcal{W}_*(X,\lambda),\partial^{\lambda})$. 
   \begin{align*}
    & \partial^{\lambda} \circ T_n(\theta)  = \partial^{\lambda} (  p_{\theta}\cdot  \theta )- \partial^{\lambda} ( p_{\theta}\cdot T_{n-1}(\partial\theta )) \\
      \overset{\eqref{Equ:bound-lambda-star}}{=}&  \theta  - p_{\theta}\cdot \partial  \theta - \partial^{\lambda} \Big(p_{\theta} \cdot T_{n-1}\big( \sum^n_{j=0}(-1)^j \theta|_{\partial_j\Delta^n} \big) \Big) \\
       \overset{\quad\ }{ = } & \theta  - p_{\theta}\cdot \partial\theta - \partial^{\lambda} \Big(p_{\theta} \cdot \big( \sum^n_{j=0}(-1)^j T_{n-1} (\theta|_{\partial_j\Delta^n}) \big) \Big)\\
   \overset{\ \divideontimes\ }{ = }   & 
   \theta  - p_{\theta}\cdot \partial\theta  - \sum^n_{j=0}(-1)^j c^{\xi}_j(\Delta^n)\cdot T_{n-1} (\theta|_{\partial_j\Delta^n}) + p_{\theta}\cdot\Big( \partial \big( \sum^n_{j=0}(-1)^j T_{n-1} (\theta|_{\partial_j\Delta^n}) \big)  \Big) \\
   \overset{\quad\ }{ = }    &   \theta  - p_{\theta}\cdot \partial  \theta  -  T_{n-1} ( \partial^{\lambda} \theta ) + p_{\theta}\cdot \big( \partial \circ T_{n-1}
    (\partial\theta ) \big)  \\    
     \overset{\eqref{Equ:T-Relation-1}}{=} &\,  \theta  - p_{\theta}\cdot \partial  \theta  -  T_{n-1} ( \partial^{\lambda} \theta )
      + p_{\theta}\cdot \big( 
     \partial  \theta - Sd_{\#} ( \partial\theta )  -
       T_{n-2} \circ \partial (\partial\theta)  \big) \\
     \overset{\eqref{Equ:Sd-Weight-Sing}}{=} & \,  \theta  - Sd_{\#}(\theta) -  T_{n-1}( \partial^{\lambda} \theta).
     \end{align*}
          The coefficients $c^{\xi}_j(\Delta^n)$ (see~\eqref{Equ:Boundary-Unified}) appear in $\overset{\divideontimes}{=}$ is because any $n$-simplex $\tau$ in $T_{n-1}(\partial_j\Delta^n)$ is the cone of $\partial_j\Delta^n$ with $\theta(b_{\partial_j\Delta^n})$ for some $0\leq j\leq n$, whose weight is $\xi(\partial_j\Delta^n)$; but the weight of $p_{\theta}*\tau$ is $\xi(\Delta^n)$ (see~\eqref{Equ:xi-star}).
The above identity means that $T$ is a chain homotopy between $Sd_{\#}$ and the identity on $(L\mathcal{W}_*(X,\lambda),\partial^{\lambda})$. 
 \n
 For a general weighted polyhedron $(X,\lambda)$ and an arbitrary weighted singular $n$-simplex $\theta: (|\Delta^{n}|,\lambda_{\xi}) \rightarrow (X,\lambda)$,  define the barycentric subdivision of $\theta$ by:
  $$  Sd_{\#}(\theta) := \theta_{\#} \big(Sd_{\#}(id_{(|\Delta^n|,\lambda_{\xi})}) \big),$$ 
  where $\theta_{\#}$ is the chain map induced by $\theta$
  (see~\eqref{Equ:f-chain-1}). Moreover, we can similarly define a chain homotopy $T$ between $Sd_{\#}$ and the identity on $ \mathcal{W}_*(X,\lambda)$ by
  $$ T_n(\theta) := \theta_{\#} \big(T_n(id_{(|\Delta^n|,\lambda_{\xi})})\big).$$

Furthermore, using iterated barycentric subdivisions we can construct a chain homotopy inverse for the inclusion
$\mathcal{W}_*^{\mathcal{U}}(X,\lambda) \hookrightarrow 
\mathcal{W}_*(X,\lambda)$.
The argument is the same as~\cite[Proposition 2.21]{Hatcher02}, hence omitted. The lemma is proved.
 \qed
      
      \vskip .8cm
      
   \section*{Acknowledgment}
  During the process of writing this paper, 
Lisu Wu is partially supported by  National Natural Science Foundation of China (Grant No.\,12201359) and  Natural Science Foundation of Shandong Province, China (Grant No.\,ZR2022QA028), and Li Yu is partially supported by
  National Natural Science Foundation of China (Grant No.\,11871266).

\end{document}